\documentclass[10pt]{amsart}

\usepackage{amsmath}
\usepackage{amsthm}
\usepackage{amssymb}
\usepackage{enumerate}
\usepackage{fullpage}
\usepackage[all]{xy}
\usepackage{hyperref}
\usepackage{wasysym}
\usepackage{url}
\usepackage{graphicx}

\usepackage{float}
\floatstyle{boxed}
\restylefloat{figure}

\theoremstyle{definition}
\newtheorem{thm}{Theorem}[section]
\newtheorem{definition}[thm]{Definition}
\newtheorem{lemma}[thm]{Lemma}
\newtheorem{definitionlemma}[thm]{Definition-Lemma}

\newtheorem{cor}[thm]{Corollary}
\newtheorem{prop}[thm]{Proposition}
\newtheorem{claim}[thm]{Claim}
\newtheorem{remark}[thm]{Remark}

\newtheorem{example}[thm]{Example}

\newtheorem{fact}[thm]{Fact}

\newtheorem{observation}[thm]{Observation}

\newtheorem{conjecture}[thm]{Conjecture}
\newtheorem{tournantdangereux}[thm]{Tournant Dangereux}

\renewcommand{\theequation}{\Roman{equation}}

\newcommand\N{\mathbb{N}}
\newcommand\Z{\mathbb{Z}}
\newcommand\Q{\mathbb{Q}}
\newcommand\R{\mathbb{R}}
\newcommand\M{\mathbb{M}}
\DeclareMathOperator\id{id}
\newcommand\T{\mathbb{T}}

\DeclareMathOperator\sign{sign}

\DeclareMathOperator\Th{Th}
\DeclareMathOperator\ded{ded}

\DeclareMathOperator\sded{sded}
\newcommand\GAMMA{\mathbb{M}}

\DeclareMathOperator\tp{tp}
\DeclareMathOperator\acl{acl}

\author{Allen Gehret}
\title{NIP for the Asymptotic Couple of the Field of Logarithmic Transseries}
\email{agehret2@illinois.edu}
\address{Department of Mathematics, University of Illinois at Urbana-Champaign, Urbana, Illinois 61801}
\date{\today}

\begin{document}

\maketitle

\begin{abstract}
The derivation on the differential-valued field $\mathbb{T}_{\log}$ of logarithmic transseries induces on its value group $\Gamma_{\log}$ a certain map $\psi$. The structure $\Gamma = (\Gamma_{\log},\psi)$ is a divisible asymptotic couple. In~\cite{gehret} we began a study of the first-order theory of $(\Gamma_{\log},\psi)$ where, among other things, we proved that the theory $T_{\log} = \Th(\Gamma_{\log},\psi)$ has a universal axiomatization, is model complete and admits elimination of quantifiers (QE) in a natural first-order language. In that paper we posed the question whether $T_{\log}$ has NIP (i.e., the Non-Independence Property). In this paper, we answer that question in the affirmative: $T_{\log}$ does have NIP. Our method of proof relies on a complete survey of the $1$-types of $T_{\log}$, which, in the presence of QE, is equivalent to a characterization of all simple extensions $\Gamma\langle\alpha\rangle$ of $\Gamma$. We also show that $T_{\log}$ does not have the Steinitz exchange property and we weigh in on the relationship between models of $T_{\log}$ and the so-called \emph{precontraction groups} of~\cite{kuhlmann1}.
\end{abstract}

\setcounter{tocdepth}{1}
\tableofcontents


\section{Introduction}
\label{introduction}

\medskip\noindent
In~\cite{gehret} we began a study of the model-theoretic and algebraic properties of $(\Gamma_{\log},\psi)$, the asymptotic couple of the differential-valued field $\T_{\log}$ of logarithmic transseries. This paper is intended to be its sequel. Here we give a complete survey of the space of 1-types over a model of the theory $T_{\log} = \Th(\Gamma_{\log},\psi)$ and use that to show that $T_{\log}$ has the Non-Independence Property (NIP), largely settling a question we raised in~\cite[\S 8]{gehret}.

\medskip\noindent
Throughout, $m$ and $n$ range over $\N=\{0,1,2,\dots\}$. 
As usual, $\Z$ is the ring of integers, $\Q$ is the field of rational numbers, and $\R$ is the field of real numbers. 
In this paper, like its prequel~\cite{gehret}, we study asymptotic couples such as $(\Gamma_{\log},\psi)$ as independent objects of interest, completely removed from any differential-valued fields from which they may arise.
A complete discussion of differential-valued fields such as $\T_{\log}$ and how they give rise to asymptotic couples is outside the scope of this paper.
We refer the interested reader to~\cite{mt} for the complete story as to how asymptotic couples such as $(\Gamma_{\log},\psi)$ fit into the broader ecosystem of asymptotic differential algebra.
For the reader's convenience, we begin with a definition of $(\Gamma_{\log},\psi)$, completely independent of $\T_{\log}$:

\medskip\noindent
Let $\bigoplus_n \R e_n$ be a vector space over $\R$ with basis $(e_n)$. Then $\bigoplus_n \R e_n$ can be
 made into an ordered group using the usual lexicographic order, i.e., by requiring for nonzero
$\sum_i r_i e_i$ that
\[
\sum r_i e_i >0\ \Longleftrightarrow\ r_n>0\  \text{\ for the least $n$ such that $r_n\ne 0$}.
\]
Let $\Gamma_{\log}$ be the above ordered abelian group $\bigoplus_n \R e_n$. 
It is often convenient to think of an element $\sum r_ie_i$ as the vector $(r_0,r_1,r_2,\ldots)$.
For an arbitrary ordered abelian group $\Gamma$ we set
$\Gamma^{\ne}:= \Gamma\setminus \{0\}$. 
We follow Rosenlicht~\cite{differentialvaluationII}
in taking the function 
\[
\psi: \Gamma_{\log}^{\ne} \to \Gamma_{\log}
\]
defined by
\[
(\underbrace{0,\ldots,0}_{n},\underbrace{r_n}_{\neq 0},r_{n+1},\ldots) \mapsto (\underbrace{1,\ldots,1}_{n+1},0,0,\ldots)
\]
as a new primitive, calling the pair $(\Gamma_{\log}, \psi)$ an \emph{asymptotic
couple} (the asymptotic couple of $\T_{\log}$).

\medskip\noindent
In Figure~\ref{StandardModel} we attempt to visualize the asymptotic couple $(\Gamma_{\log},\psi)$. As with any dense linear order, we can picture the underlying divisible ordered abelian group $\Gamma_{\log}$ as an infinite line stretching from left to right. Additionally we include a distinguished vertical stick to indicate the location of $0 = (0,0,0,\ldots)$. To represent the important subset $\Psi_{\log} = \psi(\Gamma_{\log}^{\neq})$, we draw a collection of vertical sticks to the right of $0$. The convergent and shrinking nature of this collection is intended to suggest that both
\begin{enumerate}[(a)]
\item the induced ordering $(\Psi_{\log},<)$ is isomorphic to that of the natural numbers $(\N,<)$, and
\item the distance between two adjacent sticks is much bigger than the distance between the next two adjacent sticks.
\end{enumerate}
Indeed, the difference between, say, the first and second elements of $\Psi_{\log}$ is \[
(1,1,0,\ldots)-(1,0,\ldots) = (0,1,0,\ldots)
\]
 which is infinitely larger (i.e., is a member of a larger \emph{archimedean class}, a notion defined in~\ref{orderedsetconventions} below) than the difference between the second and third elements of $\Psi_{\log}$, which is 
 \[
 (1,1,1,0,\ldots) - (1,1,0,\ldots) = (0,0,1,0,\ldots).
 \]

\begin{figure}[!htbp]
\centering
\caption{Illustration of $(\Gamma_{\log},\psi)$}
\label{StandardModel}
\resizebox{17cm}{!}{
\includegraphics{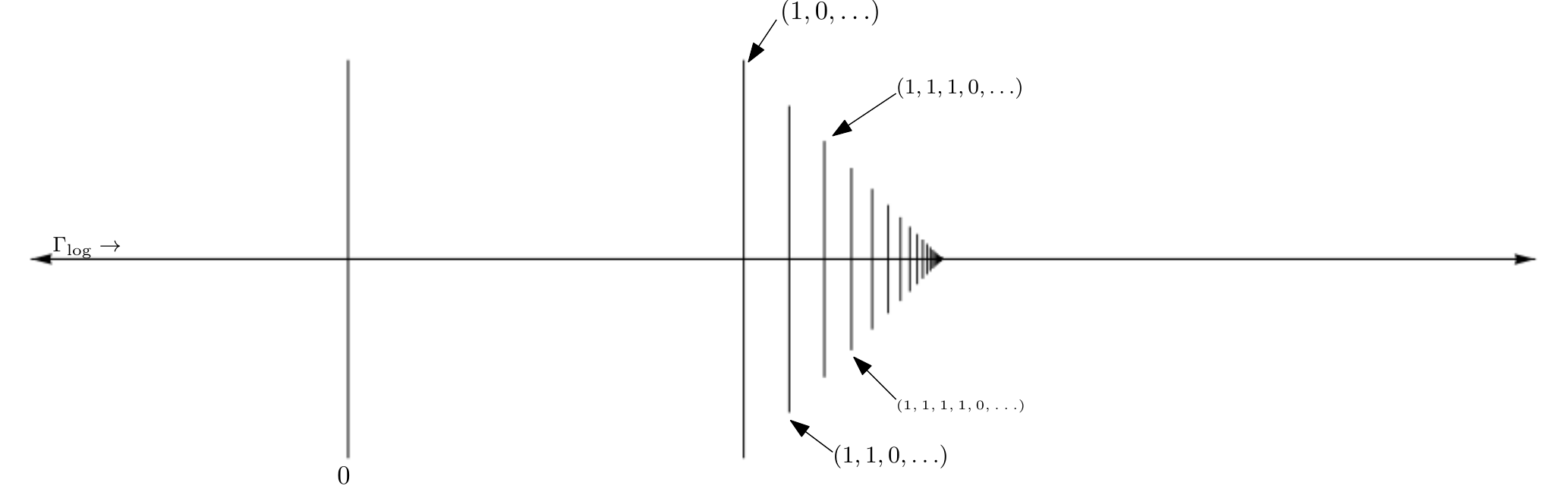}
}
\end{figure}

\noindent
Most of our intuition for this structure and its elementary extensions comes from drawing pictures of this form (for example, see Figure~\ref{TlogElementaryExt}). Our choice of drawing the infinite set $\Psi_{\log}$ in this way was inspired by the illustrations from~\cite[Ch. 10]{conway}.

\medskip\noindent
In~\cite{gehret} we gave a complete axiomatization for the first-order theory $\Th(\Gamma_{\log},\psi)$ and proved that it is model complete; see Definition~\ref{T0def} below. 
This followed from exhibiting quantifier elimination for $\Th(\Gamma_{\log},\psi)$ in a natural language $L_{\log}$ which we recall in Section~\ref{simpleextensions}. 
Finally, we showed that the discrete subset $\Psi_{\log}$ is stably embedded in the structure $(\Gamma,\psi)$.
In this paper, we continue our study of the model-theory of $(\Gamma_{\log},\psi)$ by demonstrating that it has NIP, a form of model-theoretic tameness.

\medskip\noindent
In Section~\ref{moreasymptoticintegration}, we recall from~\cite{gehret} some definitions and elementary properties relating to $H$-asymptotic couples and we introduce a few preliminary ideas mostly in the generality of divisible $H$-asymptotic couples with asymptotic integration, construed as $L_{AC}$-structures, where $L_{AC}$ is the natural language of asymptotic couples. This section can be viewed as a continuation of Sections 3 and 4 from~\cite{gehret}. The main idea from this section to be used later is Lemma~\ref{rhoZ}, a new embedding lemma that adds transfinitely many ``copies of $\Z$'' to an existing $\Psi$-set. The fact that one can do the construction as in Lemma~\ref{rhoZ} is already apparent from~\cite[Lemmas 4.11, 4.12]{gehret}, but we make this construction explicit because of its utility in classifying simple extensions in Section~\ref{sectionsimpleextensions}.

\medskip\noindent
In Section~\ref{sectionsimpleextensions}, we specialize to models of $T_{\log} = \Th(\Gamma_{\log},\psi)$ in an enriched language $L_{\log}$. There we prove Theorem~\ref{simpleextensions} which gives all the possibilities for the isomorphism types of simple extensions $\Gamma\langle\alpha\rangle$ for models $\Gamma\models T_{\log}$. In the presence of quantifier elimination, this is the same thing as giving all the possibilities for 1-types. Roughly speaking, we show that all simple extensions are controlled by at most countably many Dedekind cuts of a certain form in the set $\Psi$.

\medskip\noindent
In Section~\ref{examples}, we give explicit examples of the various possibilities of simple extensions mentioned in Theorem~\ref{simpleextensions}. This shows that Theorem~\ref{simpleextensions} doesn't merely place a bound on the possibilities of simple extensions, but really does give precisely those simple extensions that actually occur.

\medskip\noindent
In Section~\ref{countingtypes} we derive Corollary~\ref{countingtypescor} from Theorem~\ref{simpleextensions} which says that the number of 1-types over a model of size $\kappa$ is bounded by the cardinal $\ded(\kappa)^{\aleph_0}$ (where $\ded(\kappa)$ is defined in~\ref{settheoryconventions} below).
 
\medskip\noindent
In Section~\ref{NIPsettheory}, we give the definition of the model-theoretic notion of NIP and prove NIP for $T_{\log}$ using a counting-types and absoluteness swindle.  It is a fact that theories with the independence property (IP) always have $2^{\kappa}$-many 1-types over a model of size $\kappa$. By a forcing result of Mitchell~\cite{mitchell}, it is consistent with ZFC that $\ded(\kappa)^{\aleph_0}<2^{\kappa}$ for some cardinal $\kappa$ and so the theory $T_{\log}$ must have NIP. Our basic references for NIP are~\cite{simonNIP} and~\cite{adlerNIP}.

\medskip\noindent
In Section~\ref{otherresults}, we tie up some loose ends and raise an additional question. In particular, we show that the theory of $T_{\log}$ does not have the so-called \emph{Steinitz exchange property}. This follows from the ideas in Section~\ref{sectionsimpleextensions}. We also demonstrate a way to produce new $\psi$-maps given a divisible $H$-asymptotic couple $(\Gamma,\psi)$ with asymptotic integration. Finally, we weigh in on the relationship between divisible $H$-asymptotic couples with asymptotic integration and the divisible precontraction groups of Kuhlmann (see~\cite{kuhlmann1,kuhlmann2}). In parallel with~\cite[\S 5]{someremarks}, we show that it is impossible to definably reconstruct the $\psi$-map of a model of $T_{\log}$ from the underlying precontraction group.

\medskip\noindent
Finally, in Section~\ref{conclusion} we give a list of remaining questions and issues.

\subsection{Set Theory Conventions}
\label{settheoryconventions}

We assume the reader is familiar with the basic concepts and definitions from set theory (for example, see~\cite{kunen} or~\cite{jech}). Throughout, $\kappa,\lambda$ will denote infinite cardinals and $\eta,\nu$ will denote (possibly finite) ordinals. We define 
\[
\ded(\kappa) := \sup\{\lambda:\text{there is a linear order of size $\lambda$ which has a dense subset of size $\kappa$}\},
\]
where \emph{dense} is in the sense of the usual order topology.
In general we have that $\kappa<\ded(\kappa)\leq \ded(\kappa)^{\aleph_0}\leq 2^{\kappa}$ for all $\kappa$ with equality if $\kappa = \aleph_0$. Furthermore, $\ded(\kappa)\leq \ded(\lambda)$ if $\kappa\leq\lambda$.

\subsection{Ordered Set Conventions}
\label{orderedsetconventions}

By  ``ordered set'' we mean ``totally ordered set''.

\medskip\noindent
Let $S$ be an ordered set. Below, the ordering on $S$ will be denoted by $\leq$, and a subset of $S$ is viewed as ordered by the induced ordering. We put $S_{\infty}:= S\cup\{\infty\}$, $\infty\not\in S$, with the ordering on $S$ extended to a (total) ordering on $S_{\infty}$ by $S<\infty$.  Suppose that $B$ is a subset of an ordered set extending  $S$. We put $S^{>B}:=\{s\in S:s>b\text{ for every $b\in B$}\}$ and we denote $S^{>\{a\}}$ as just $S^{>a}$; similarly for $\geq, <,$ and $\leq$ instead of $>$. For $a,b\in S\cup\{\infty\}$ and $B\subseteq S$ we put
\[
[a,b]_{B}:=\{x\in B: a\leq x\leq b\}.
\]
If $B=S$, then we usually write $[a,b]$ instead of $[a,b]_S$. Given subsets $S_0,S_1\subseteq S$, we say the pair $(S_0,S_1)$ is a \textbf{cut in} $S$, if $S_0 = S^{<S_1}$ and $S_1=S^{>S_0}$ and we say that an element $x$ of an ordered set extending $S$ \textbf{realizes the cut} $(S_0,S_1)$ if $S_0 = S^{<x}$ and $S_1 = S^{>x}$. We say that $S$ is a \textbf{successor set} if every element $x\in S$ has an \textbf{immediate successor} $y\in S$, that is, $x<y$ and for all $z\in S$, if $x<z$, then $y\leq z$. For example, $\N$ and $\Z$ with their usual ordering are successor sets.

\medskip\noindent
We say that $S$ is a \textbf{copy of $\Z$} (respectively, \textbf{copy of $\N$}) if $(S,<)$ is isomorphic to $(\Z,<)$ (respectively, $(\N,<)$).

\medskip\noindent
Suppose that $G$ is an ordered abelian group. Then we set $G^{\neq}:=G\setminus\{0\}$, $G^{<}:= G^{<0}$ and $G^{>}:= G^{>0}$. We define $|g| := \max\{g,-g\}$ for $g\in G$. For $a\in G$, the \textbf{archimedean class} of $a$ is defined by
\[
[a] := \{g\in G: |a|\leq n|g|\text{ and }|g|\leq n|a|\text{ for some }n\geq 1\}.
\]
The archimedean classes partition $G$. Each archimedean class $[a]$ with $a\neq 0$ is the disjoint union of the two convex sets $[a]\cap G^{<}$ and $[a]\cap G^{>}$. We order the set $[G]:=\{[a]:a\in G\}$ of archimedean classes by
\[
[a]<[b] :\Longleftrightarrow n|a|<|b|\text{ for all }n\geq 1.
\]
We have $[0]<[a]$ for all $a\in G^{\neq}$, and
\[
[a]\leq[b] :\Longleftrightarrow |a|\leq n|b| \text{ for some } n\geq 1.
\]
We say that $G$ is \textbf{archimedean} if $[G^{\neq}]:=[G]\setminus\{[0]\}$ is a singleton.

\subsection{Model Theory Conventions}

Throughout $L$ will denote a one-sorted language and $T$ will be a complete $L$-theory with infinite models. 
We will work in this general setting when discussing model-theoretic issues (such as NIP). 
We will often consider a model $\M\models T$ and a cardinal $\kappa(\M)>|L|$ such that $\M$ is $\kappa(\M)$-saturated and strongly $\kappa(\M)$-homogeneous. 
Such a model is called a \textbf{monster model} of $T$. 
In particular, every model of $T$ of size $\leq \kappa(\M)$ has an elementary embedding into $\M$. 
``Small'' will mean ``of size $<\kappa(\M)$''. 
$A$ will always denote a \emph{small} parameter set in $\M$. 
If $M$ is a parameter set underlying an elementary submodel of $\M$, then we denote this elementary submodel also by $M$. 
For a parameter set $A$, we let $\langle A\rangle$ denote the $L$-substructure of $\M$ generated by $A$. 
Similarly we let $M\langle A\rangle$ denote $\langle M\cup A\rangle$. 
Note that if $T$ has a universal axiomatization and is model complete, then $\langle A\rangle$ is always a small elementary substructure of $\M$.
We let $S^n(A)$ denote the space of $n$-types over $A$.

\section{More Asymptotic Integration}
\label{moreasymptoticintegration}

\subsection{Asymptotic Couples}
In general, an \textbf{asymptotic couple} is a pair
$(\Gamma, \psi)$ where $\Gamma$ is an ordered abelian group and
$\psi: \Gamma^{\ne} \to \Gamma$ satisfies for all $\alpha,\beta\in\Gamma^{\neq}$,
\begin{itemize}
\item[(AC1)] $\alpha+\beta\neq 0 \Longrightarrow \psi(\alpha+\beta)\geq \min(\psi(\alpha),\psi(\beta))$;
\item[(AC2)] $\psi(r\alpha) = \psi(\alpha)$ for all $r\in\Z^{\neq}$, in particular, $\psi(-\alpha) = \psi(\alpha)$;
\item[(AC3)] $\alpha>0 \Longrightarrow \alpha+\psi(\alpha)>\psi(\beta)$.
\end{itemize}
If in addition for all $\alpha,\beta\in\Gamma$,
\begin{itemize}
\item[(HC)] $0<\alpha\leq\beta\Rightarrow \psi(\alpha)\geq \psi(\beta)$,
\end{itemize}
then $(\Gamma,\psi)$ is said to be of \textbf{$H$-type}, or to be an \textbf{$H$-asymptotic couple}

\medskip\noindent
The primary example of an $H$-asymptotic couple is the object $(\Gamma_{\log},\psi)$ defined in Section~\ref{introduction}. Asymptotic couples were introduced by Rosenlicht in~\cite{differentialvaluation1,differentialvaluations,differentialvaluationII} to study \emph{differential-valued fields}. The prefix $H$ in ``$H$-asymptotic couple'' is in honor of the pioneers of the subject: Borel, Hahn, Hardy, and Hausdorff.

\medskip\noindent
Asymptotic couples commonly show up in nature as the value groups of certain kinds of valued differential fields (the so-called \emph{asymptotic fields}). In this case, the map $\psi:\Gamma^{\neq}\to\Gamma$ is induced by the logarithmic derivative on the field and the map $\id+\psi:\Gamma^{\neq}\to\Gamma$ is induced by the derivative. This is the motivation for the terminology ``asymptotic integration'' as well as the notations $\alpha^{\dagger}$ and $\alpha'$ (all introduced below). For the complete story see~\cite{mt}.

\medskip\noindent
\emph{For the rest of this subsection $(\Gamma,\psi)$ will be an arbitrary asymptotic couple and $\alpha,\beta$ will range over $\Gamma$.} By convention we extend $\psi$ to all of $\Gamma$ by setting $\psi(0):=\infty$. Then $\psi(\alpha+\beta)\geq\min(\psi(\alpha),\psi(\beta))$ holds for all $\alpha,\beta\in\Gamma$, and construe $\psi:\Gamma\to\Gamma_{\infty}$ as a (non-surjective) valuation on the abelian group $\Gamma$. If $(\Gamma,\psi)$ is of $H$-type, then this valuation is convex. The following property of valuations is immediate and will be used often:

\begin{fact}
\label{valuationfact}
If $\psi(\alpha)<\psi(\beta)$, then $\psi(\alpha+\beta) = \psi(\alpha)$.
\end{fact}

\medskip\noindent
Let $L_{AC}$ be the natural language of asymptotic couples; $L_{AC} = \{0,+,-,<,\psi,\infty\}$ where $0,\infty$ are constant symbols, $+$ is a binary function symbol, $-, \psi$ are unary function symbols and $<$ is a binary relation symbol. We consider an asymptotic couple $(\Gamma,\psi)$ as an $L_{AC}$-structure with underlying set $\Gamma_{\infty}$ and the obvious interpretation of the symbols of $L_{AC}$, with $\infty$ as a default value:
\[
-\infty = \gamma+\infty = \infty+\gamma=\infty+\infty = \psi(0) = \psi(\infty) = \infty
\]
for all $\gamma\in\Gamma$.

\medskip\noindent
For $\alpha\in\Gamma^{\neq}$ we shall also use the following notation:
\[
\alpha^{\dagger}:= \psi(\alpha), \quad \alpha':=\alpha+\psi(\alpha).
\]
The following subsets of $\Gamma$ play special roles:
\[
(\Gamma^{\neq})' := \{\gamma':\gamma\in\Gamma^{\neq}\}, \quad (\Gamma^{>})' := \{\gamma':\gamma\in\Gamma^{>}\},
\]
\[
\Psi := \psi(\Gamma^{\neq}) = \{\gamma^{\dagger}:\gamma\in\Gamma^{\neq}\} = \{\gamma^{\dagger}:\gamma\in\Gamma^{>}\}.
\]

\noindent
For an arbitrary asymptotic couple $(\Gamma',\psi')$ we may occasionally refer to the set $\Psi_{\Gamma'}:=\psi'((\Gamma')^{\neq})$ as ``the $\Psi$-set of $(\Gamma',\psi')$'' and to the function $\psi'$ as ``the $\psi$-map of $(\Gamma',\psi')$''.

\medskip\noindent
Note that by (AC3) we have $\Psi<(\Gamma^{>})'$. It is also the case that $(\Gamma^{<})'<(\Gamma^{>})'$:

\begin{lemma}
The map $\gamma\mapsto \gamma' = \gamma+\psi(\gamma):\Gamma^{\neq}\to\Gamma^{\neq}$ is strictly increasing. In particular:
\begin{enumerate}
\item $(\Gamma^{<})'<(\Gamma^{>})'$, and
\item for $\beta\in\Gamma$ there is at most one $\alpha\in\Gamma^{\neq}$ such that $\alpha' = \beta$.
\end{enumerate}
\end{lemma}
\begin{proof}
This follows from~\cite[Lemma 6.5.4(iii)]{mt}.
\end{proof}

\medskip\noindent
We say that an asymptotic couple $(\Gamma,\psi)$ has \textbf{asymptotic integration} if
\[
\Gamma = (\Gamma^{\neq})'.
\]
The primary example of an asymptotic couple with asymptotic integration is $(\Gamma_{\log},\psi)$.

\subsection{Asymptotic Integration}
\emph{In this subsection $(\Gamma,\psi)$ will be an arbitrary divisible $H$-asymptotic couple with asymptotic integration.}
We will construe $(\Gamma,\psi)$ as an $L_{AC}$-structure. Asymptotic integration allows us to define the functions $\int, s,$ and $\chi$ on $\Gamma$:

\begin{definition}
\label{functionsdefs}
For $\alpha\in\Gamma$ we let $\int\alpha$ denote the unique element $\beta\in\Gamma^{\neq}$ such that $\beta'=\alpha$ and we call $\beta = \int\alpha$ the \textbf{integral} of $\alpha$. This gives us a function $\int:\Gamma\to\Gamma^{\neq}$ which is the inverse of $\gamma\mapsto\gamma':\Gamma^{\neq}\to\Gamma$. We define the \textbf{successor function} $s:\Gamma\to\Psi$ by $\alpha\mapsto \psi(\int\alpha)$. Finally, we define the \textbf{contraction map} $\chi:\Gamma^{<}\to\Gamma^{<}$ by $\alpha\mapsto \int\psi(\alpha)$.
\end{definition}

\begin{example}
\label{functionformulas}
For the asymptotic couple $(\Gamma_{\log},\psi)$ defined in Section~\ref{introduction}, we give explicit formulas for the integral and successor functions in~\cite[Examples 2.9 and 3.10]{gehret}. For the reader's convenience we restate them here and also give the formula for the contraction map:
\begin{enumerate}
\item (Integral) For $\alpha = (r_0,r_1,r_2,\ldots)\in\Gamma_{\log}$, take the unique $n$ such that $r_n\neq 1$ and $r_m=1$ for $m<n$. Then the formula for $\alpha\mapsto \int\alpha$ is given as follows:
\[
\alpha = (\underbrace{1,\ldots,1}_n,\underbrace{r_n}_{\neq 1},r_{n+1},r_{n+2}\ldots) \mapsto \textstyle{\int}\alpha = (\underbrace{0,\ldots,0}_{n},r_n-1,r_{n+1},r_{n+2},\ldots):\Gamma_{\log} \to \Gamma_{\log}^{\neq}
\]
\item (Successor) For $\alpha = (r_0,r_1,r_2,\ldots)\in\Gamma_{\log}$, take the unique $n$ such that $r_n\neq 1$ and $r_m = 1$ for $m<n$. Then the formula for $\alpha\mapsto s(\alpha)$ is given as follows:
\[
\alpha = (\underbrace{1,\ldots,1}_n,\underbrace{r_n}_{\neq 1},r_{n+1},r_{n+1}\ldots) \mapsto s(\alpha) = (\underbrace{1,\ldots,1}_{n+1},0,0,\ldots):\Gamma_{\log}\to\Psi_{\log}\subseteq\Gamma_{\log}
\]
\item (Contraction) For $\alpha = (r_0,r_1,r_2,\ldots)\in\Gamma_{\log}^{<}$, take the unique $n$ such that $r_n< 0$ and $r_k = 0$ for $k<n$. Then the formula for $\alpha\mapsto \chi(\alpha)$ is given as follows:
\[
\alpha = (\underbrace{0,\ldots,0}_{n},\underbrace{r_n}_{< 0},r_{n+1},\ldots)\mapsto \chi(\alpha) = (\underbrace{0,\ldots,0}_{n+1},-1,0,0,\ldots):\Gamma^{<}_{\log}\to\Gamma^{<}_{\log}
\]
\end{enumerate}
\end{example}

\medskip\noindent
To get a feel for how the functions $\int,s$, and $\chi$ behave in general, we record here some of their elementary properties.

\begin{lemma}\label{functionproperties} For all $\alpha,\beta\in\Gamma$:
\begin{enumerate}
\item $\int\alpha = \alpha-s\alpha$;
\item $\alpha\in (\Gamma^{<})' \Longrightarrow \alpha<s\alpha$;
\item $\alpha\in (\Gamma^{>})'\Longrightarrow \alpha>s\alpha$;
\item $\alpha<\beta<(\Gamma^{>})'\Longrightarrow s\alpha\leq s\beta$;
\item $(\Gamma^{<})'<\alpha<\beta\Longrightarrow s\alpha\geq s\beta$;
\item $\beta = \psi(\alpha-\beta)$ iff $\beta = s(\alpha)$;
\item $\alpha<\beta<0\Longrightarrow \chi(\alpha)\leq\chi(\beta)$;
\item $\alpha<0\Longrightarrow [\alpha]>[\chi(\alpha)]$;
\item $\alpha<0\Longrightarrow \chi(\alpha)+\psi(\chi(\alpha)) = \psi(\alpha)$.
\end{enumerate}
\end{lemma}
\begin{proof}
(2) is~\cite[Lemma 3.3]{gehret}, (6) is~\cite[Lemma 3.7]{gehret}, and (8) is~\cite[Lemma 9.2.18(iii)]{mt}. The rest follow easily from the definitions and previously stated properties of H-asymptotic couples.
\end{proof}

\medskip\noindent
The primary $L_{AC}$-theory of interest is $T_0$:
\begin{definition}
\label{T0def}
Let $T_0$ be the $L_{AC}$-theory whose models are the divisible $H$-asymptotic couples with asymptotic integration such that
\begin{itemize}
\item $\Psi$ as an ordered subset of $\Gamma$ has a least element $s0$,
\item $s0>0$,
\item $\Psi$ as an ordered subset of $\Gamma$ is a successor set,
\item for each $\alpha\in\Psi$, the immediate successor of $\alpha$ in $\Psi$ is $s\alpha$, and
\item $\gamma\mapsto s\gamma:\Psi\to\Psi^{>s0}$ is a bijection.
\end{itemize}
\end{definition}

\noindent
In~\cite{gehret} we showed that the $L_{AC}$-theory $T_0$ is complete and model complete. In particular, $T_0 = \Th_{L_{AC}}(\Gamma_{\log},\psi)$. We also showed that for models $(\Gamma,\psi)$ of $T_0$, the set $\Psi$ is stably embedded in $(\Gamma,\psi)$.

\medskip\noindent
In Figure~\ref{TlogElementaryExt} we illustrate a ``typical'' model of $T_0$. Here the set $\Psi$ no longer has order type $(\N,<)$, but in fact has the order type of $(\N,<)$ followed by copies of $(\Z,<)$. Here the copies of $(\Z,<)$ are indexed by the linear order $(\N,<)$, but in general the copies of $(\Z,<)$ may be indexed by any linear order. This is clear because the ordered set $(\Psi,<)$ is elementarily equivalent to the ordered set $(\N,<)$. The dashed line located at ``$\sup\Psi$'' serves to indicate the boundary between $(\Gamma^{<})'$ and $(\Gamma^{>})'$. In particular, $(\Gamma^{>})' = \Gamma^{>\Psi}$ and $(\Gamma^{<})' = \Psi^{\downarrow}$ (the downward closure of the set $\Psi$ in $\Gamma$). The function $s:\Gamma\to\Psi$ is defined on all of $\Gamma$, but we illustrate here that its restriction to $\Psi$ really does make it an actual successor function $\gamma\mapsto s\gamma:\Psi\to\Psi^{>s0}$. Finally, for the sake of completeness, we have included the function $p$ in this illustration. The function $p$ is defined to be the inverse to $\gamma\mapsto s\gamma:\Psi\to\Psi^{>s0}$, and we extend it to a function on the rest of $\Gamma_{\infty}$ by having it take the value $\infty$ everywhere else. We will formally add $s$ and $p$ to our language in Section~\ref{sectionsimpleextensions}, but we include them here because they are definable in models of $T_0$.

\begin{figure}[!htbp]
\centering
\caption{A typical model of $T_0$}
\label{TlogElementaryExt}
\resizebox{17cm}{!}{
\includegraphics{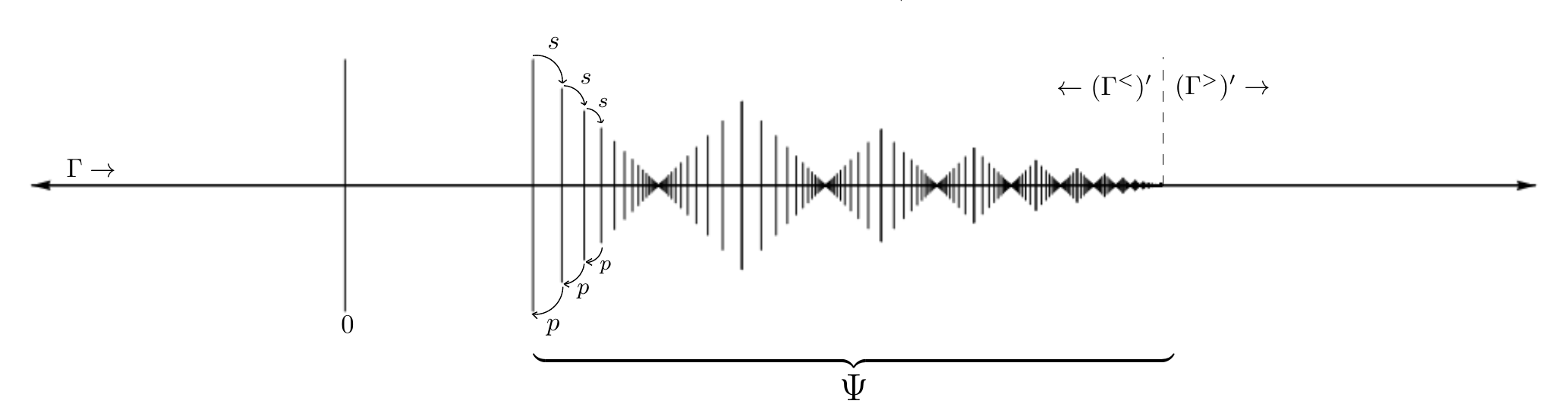}
}
\end{figure}

\medskip\noindent
For the rest of this section we continue with our standing assumption that $(\Gamma,\psi)$ is an arbitrary divisible $H$-asymptotic couple with asymptotic integration. However, it may be useful for the reader to keep in mind the specific case when $(\Gamma,\psi)\models T_0$.

\begin{definition}
We say $B\subseteq\Psi$ is an \textbf{$s$-cut of $\Psi$} if $B$ is an \emph{upward} closed subset of $\Psi$ such that $s(\Psi\setminus B)\subseteq (\Psi\setminus B)$. Let $\sded(\Psi)$ be the collection of all $s$-cuts of $\Psi$. We define a linear ordering $\leq$ on $\sded(\Psi)$ by $B_0\leq B_1$ iff $B_0\supseteq B_1$.
\end{definition}


\begin{remark}
We defined $s$-cuts here as ``right cuts'' only for notational convenience in Lemma~\ref{rhoZ}. Given an $s$-cut $B$ of $\Psi$, we identify it with the cut $(\Psi^{<B},B)$ in $\Psi$.
\end{remark}

\begin{definition}
For $\alpha,\beta\in\Psi$, we define $\alpha \ll \beta$ to mean $s^n\alpha<\beta$ for all $n$, and define $\alpha \gg\beta$ to mean $\beta\ll \alpha$. It follows that if $\alpha\ll\beta$, then there is a $B\in\sded(\Psi)$ such that $\alpha< B\ni\beta$. Finally, we define the equivalence relation $\sim_s$ on $\Psi$:
\[
\alpha\sim_s\beta :\Longleftrightarrow \text{$\alpha\not\ll\beta$ and $\beta\not\ll\alpha$}
\]
and we call the equivalence class $\alpha/\sim_s$ of $\alpha$ the \textbf{$s$-class} of $\alpha$. If $(\Gamma,\psi)\models T_0$, then the $s$-class of $\alpha$ is thought of as the copy of $\Z$ or initial copy of $\N$ that $\alpha$ lives on.
\end{definition}

\noindent
For divisible $H$-asymptotic couples with asymptotic integration, it is useful to have the following stratification in mind:
\[
\xymatrix{
\Gamma^{\neq}\ar@{->>}^-{\gamma\mapsto [\gamma]}[d] \\
[\Gamma^{\neq}]\ar@{->>}^-{[\gamma]\mapsto \psi(\gamma)}[d] &\text{archimedean classes}\\
\Psi\ar@{->>}^-{\psi(\gamma)\mapsto \psi(\gamma)/\sim_s}[d] &\text{$\Psi$-set}\\
\Psi/\sim_s &\text{$s$-classes on the $\Psi$-set}\\
}
\]

\medskip\noindent
\begin{definition}
Let $(\Gamma,\psi)$ and $(\Gamma_1,\psi_1)$ be asymptotic couples. An \textbf{embedding}
\[
h:(\Gamma,\psi)\to(\Gamma_1,\psi_1)
\]
is an embedding $h:\Gamma\to\Gamma_1$ of ordered abelian groups such that
\[
h(\psi(\gamma)) = \psi_1(h(\gamma))\text{ for $\gamma\in\Gamma^{\neq}$.}
\]
If $\Gamma\subseteq\Gamma_1$ and the inclusion $\Gamma\to\Gamma_1$ is an embedding $(\Gamma,\psi)\to(\Gamma_1,\psi_1)$, then we call $(\Gamma_1,\psi_1)$ an \textbf{extension} of $(\Gamma,\psi)$, and we also indicate this by $(\Gamma,\psi)\subseteq(\Gamma_1,\psi_1)$.
\end{definition}

\medskip\noindent
The proof of quantifier elimination for the theory of $(\Gamma_{\log},\psi)$ in~\cite{gehret} is built upon an arsenal of embedding lemmas for divisible H-asymptotic couples (not necessarily with asymptotic integration).
For the purposes of the current section, we only need to recall the following embedding lemma~\cite[Lemmas 4.11 and 4.12]{gehret} which adds a single copy of $\Z$ to $\Psi$ (and adds a single point to $\Psi/\sim_s$):

\begin{lemma}
\label{Zinmiddle}Let $B\in\sded(\Psi)$ be such that $B\neq\Psi$. Then there is a divisible $H$-asymptotic couple $(\Gamma_B,\psi_B)\supseteq (\Gamma,\psi)$ with a family $(\beta_k)_{k\in\Z}$ in $\Psi_B$ satisfying the following conditions:
\begin{enumerate}
\item $(\Gamma_B,\psi_B)$ has asymptotic integration;
\item $\Gamma^{<B}<\beta_k<B$, and $s_B(\beta_k) = \beta_{k+1}$ for all $k$;
\item $\Psi_B = \Psi\cup\{\beta_k:k\in\Z\}$;
\item for any embedding $i:(\Gamma,\psi)\to (\Gamma^*,\psi^*)$ into a divisible $H$-asymptotic couple with asymptotic integration and any family $(\beta_k^*)_{k\in\Z}$ in $\Psi^*$ such that $i(\Gamma^{<B})<\beta^*_k<i(B)$ and $s^*(\beta^*_k) = \beta^*_{k+1}$ for all $k$, there is a unique extension of $i$ to an embedding $(\Gamma_B,\psi_B)\to(\Gamma^*,\psi^*)$ sending $\beta_k$ to $\beta^*_k$ for all $k$;
\item if $(\Gamma,\psi)$ is a model of $T_0$, then so is $(\Gamma_B,\psi_B)$.
\end{enumerate}
\end{lemma}
\begin{proof}[Sketch of Proof]
The underlying abelian group of the extension $(\Gamma_B,\psi_B)$ will be $\Gamma_B:= \Gamma\oplus\bigoplus_{k\in\Z}\Q\beta_k$. The ordering and $\psi$-map are then defined on $\Gamma_B$ in such a way as to guarantee that (4) holds. If $B=\emptyset$, then this is~\cite[Lemma 4.11]{gehret}, and if $B\neq \emptyset$, then this is~\cite[Lemma 4.12]{gehret}.
\end{proof}

\medskip\noindent
In Figure~\ref{ZinSCutIPE}, we illustrate an instance of the construction that is done in Lemma~\ref{Zinmiddle} (over a model of $T_0$). Technically speaking, here $B$ (as a set) is the two rightmost copies of $\Z$, however, we think of $B$ as indicating the cut between existing copies of $\Z$ where a new copy of $\Z$ (namely, $(\beta_k)_{k\in\Z}$) is to be added. 


\begin{figure}[!htbp]
\centering
\caption{Example of Lemma~\ref{Zinmiddle} in action}
\label{ZinSCutIPE}
\resizebox{17cm}{!}{
\includegraphics{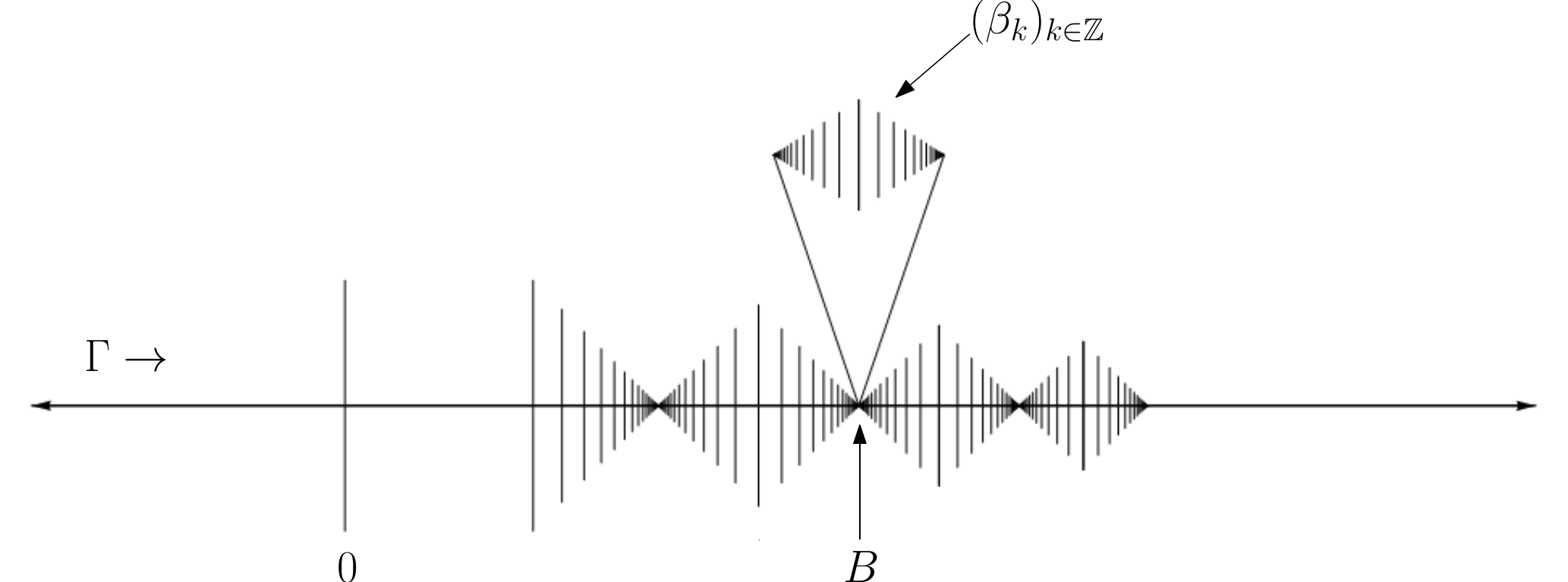}
}
\end{figure}

\medskip\noindent
As for the universal property, suppose $i:(\Gamma,\psi)\to(\Gamma^*,\psi^*)$ is an embedding as in (4) from Lemma~\ref{Zinmiddle} above. The uniqueness of the extension of $i$ to an embedding $(\Gamma_B,\psi_B)\to (\Gamma^*,\psi^*)$ depends heavily on the specification of the family $(\beta_k^*)_{k\in\Z}$ in $\Psi^*$ and in particular the requirement that $\beta_k\mapsto \beta_k^*$ for all $k$:
\[
\xymatrix{
&&\text{$(\Gamma^*,\psi^*)$ \textbf{and} $(\beta_k^*)_{k\in\Z}$} \\
\text{$(\Gamma_B,\psi_B)$ \textbf{and} $(\beta_k)_{k\in\Z}$} \ar^{\exists!}@{.>}[urr]&& \\
&&\\
(\Gamma,\psi)\ar_-{i}[uuurr] \ar@{-}[uu]&&\\
}
\]
In fact, if we were to drop the requirement that the extension of $i$ to an embedding $(\Gamma_B,\psi_B)\to (\Gamma^*,\psi^*)$ has the property that $\beta_k\mapsto \beta_k^*$ for all $k\in\Z$, then there would always be infinitely many distinct extensions of $i$ to embeddings $(\Gamma_B,\psi_B)\to (\Gamma^*,\psi^*)$:
\[
\xymatrix{
&&\text{$(\Gamma^*,\psi^*)$} \\
\text{$(\Gamma_B,\psi_B)$} \ar^{\exists^{\infty}}@{.>}[urr]&& \\
&&\\
(\Gamma,\psi)\ar_-{i}[uuurr] \ar@{-}[uu]&&\\
}
\]
This follows from Lemma~\ref{Zinmiddle} by considering the reindexing $(\beta_{k+l}^*)_{k\in\Z}$ of the family $(\beta_{k}^*)_{k\in\Z}$ by an arbitrary $l\in\Z$.

\medskip\noindent
 In the lemma below we add transfinitely many copies of $\Z$ to $\Psi$. We think of the extension $(\Gamma_{\rho},\psi_{\rho})$ of $(\Gamma,\psi)$ constructed in that lemma as adding $\nu$-many copies of $\Z$ to $\Psi$ in the $s$-cuts specified by $\rho$.

\begin{lemma}
\label{rhoZ}
Let $\rho:\nu\to\sded(\Psi)\setminus\{\Psi\}$ be an increasing function. Then there is a divisible $H$-asymptotic couple $(\Gamma_{\rho},\psi_{\rho})\supseteq (\Gamma,\psi)$ with a family $(\beta_{k,\eta})_{k\in\Z,\eta<\nu}$ in $\Psi_{\rho}$ satisfying the following conditions:
\begin{enumerate}
\item $(\Gamma_{\rho},\psi_{\rho})$ has asymptotic integration;
\item $\Gamma^{<\rho(\eta)}<\beta_{k,\eta}<\rho(\eta)$, and $s_{\rho}(\beta_{k,\eta}) = \beta_{k+1,\eta}$ for all $k\in\Z$ and $\eta<\nu$;
\item $\beta_{k,\eta_0}<\beta_{l,\eta_1}$ for all $k,l\in\Z$ and $\eta_0<\eta_1<\nu$;
\item $\Psi_{\rho} = \Psi\cup\{\beta_{k,\eta}:k\in\Z,\eta<\nu\}$;
\item for any embedding $i:(\Gamma,\psi)\to(\Gamma^*,\psi^*)$ into a divisible $H$-asymptotic couple with asymptotic integration and any family $(\beta_{k,\eta}^*)_{k\in\Z,\eta<\nu}$ in $\Psi^*$ such that $i(\Gamma^{<\rho(\eta)})<\beta_{k,\eta}^*<i(\rho(\eta))$ and $s^*(\beta^*_{k,\eta}) = \beta^*_{k+1,\eta}$ for all $k\in\Z$ and $\eta<\nu$, and $\beta^*_{k,\eta_0}<\beta^*_{l,\eta_1}$ for all $k,l\in\Z$ and $\eta_0<\eta_1<\nu$,
 then there is a unique extension of $i$ to an embedding $(\Gamma_{\rho},\psi_{\rho})\to(\Gamma^*,\psi^*)$ sending $\beta_{k,\eta}$ to $\beta_{k,\eta}^*$ for all $k\in\Z$ and $\eta<\nu$;
\item if $(\Gamma,\psi)$ is a model of $T_0$, then so is $(\Gamma_{\rho},\psi_{\rho})$.
\end{enumerate}
\end{lemma}
\begin{proof}
We will prove this by transfinite induction on $\nu$.

($\nu = 0$) In this case we set $(\Gamma_{\rho},\psi_{\rho}):=(\Gamma,\psi)$ and we are done.

($\nu=\eta+1$) By the inductive hypothesis, we can construct an extension $(\Gamma_{\rho\upharpoonright\eta},\psi_{\rho\upharpoonright\eta})$ of $(\Gamma,\psi)$ which satisfies properties (1)-(5) for the function $\rho\upharpoonright\eta:\eta\to\sded(\Psi)$. 

\begin{claim}
$\rho(\eta)$ is an $s$-cut in $\Psi_{\rho\upharpoonright\eta}$.
\end{claim}
\begin{proof}[Proof of claim]
By the inductive hypothesis, $\Psi_{\rho\upharpoonright\eta}= \Psi\cup\{\beta_{k,\eta_0}:k\in\Z,\eta_0<\eta\}$, so it suffices to prove that $\beta_{k,\eta_0}<\rho(\eta)$ for all $k\in\Z$ and $\eta_0<\eta$. This is clear because $\beta_{k,\eta_0}<\rho(\eta_0)$ by (3) for $(\Gamma_{\rho\upharpoonright\eta},\psi_{\rho\upharpoonright\eta})$ and $\rho(\eta_0)\leq\rho(\eta)$ because $\rho$ is increasing.
\end{proof}

Since $\rho(\eta)$ is also an $s$-cut in $\Psi_{\rho\upharpoonright\eta}$, we can use Lemma~\ref{Zinmiddle} to add a copy of $\Z$ to $(\Gamma_{\rho\upharpoonright\eta},\psi_{\rho\upharpoonright\eta})$ at $\rho(\eta)$. Thus we set $(\Gamma_{\rho},\psi_{\rho}):=((\Gamma_{\rho\upharpoonright\eta})_{\rho(\eta)},(\psi_{\rho\upharpoonright\eta})_{\rho(\eta)})$. As an extension of $(\Gamma,\psi)$, it is clear that $(\Gamma_{\rho},\psi_{\rho})$ satisfies properties (1)-(4). Property (5) is satisfied because $(\Gamma_{\rho\upharpoonright\eta},\psi_{\rho\upharpoonright\eta})$ satisfies property (5) over $(\Gamma,\psi)$ and $(\Gamma_{\rho},\psi_{\rho})$ satisfies the universal property of Lemma~\ref{Zinmiddle} over $(\Gamma_{\rho\upharpoonright\eta},\psi_{\rho\upharpoonright\eta})$.

($\nu$ limit ordinal) By the inductive hypothesis, for all $\eta_0<\eta_1<\nu$ we can construct extensions $(\Gamma_{\rho\upharpoonright\eta_i},\psi_{\rho\upharpoonright\eta_i})$ of $(\Gamma,\psi)$ ($i=1,2$) such that there is a unique embedding $i_{\eta_0,\eta_1}:(\Gamma_{\rho\upharpoonright\eta_0},\psi_{\rho\upharpoonright\eta_0})\to (\Gamma_{\rho\upharpoonright\eta_1},\psi_{\rho\upharpoonright\eta_1})$ over $(\Gamma,\psi)$ such that $\beta_{k,\eta}\mapsto\beta_{k,\eta}$ for all $k\in\Z$ and $\eta<\eta_0$. 
\[
\xymatrix{
(\Gamma_{\rho\upharpoonright\eta_0},\psi_{\rho\upharpoonright\eta_0}) \ar^{i_{\eta_0,\eta_1}}[r]& (\Gamma_{\rho\upharpoonright\eta_1},\psi_{\rho\upharpoonright\eta_1}) \\
(\Gamma,\psi) \ar[u] \ar[ur]& \\
}
\]
Thus without loss of generality we may assume that for all $\eta_0<\eta_1<\nu$ we have an increasing chain:
\[
(\Gamma,\psi)\subseteq (\Gamma_{\rho\upharpoonright\eta_0},\psi_{\rho\upharpoonright\eta_0})\subseteq  (\Gamma_{\rho\upharpoonright\eta_1},\psi_{\rho\upharpoonright\eta_1}) 
\]
Therefore we may set $(\Gamma_{\rho},\psi_{\rho}):= (\bigcup_{\eta<\nu}\Gamma_{\rho\upharpoonright\eta},\bigcup_{\eta<\nu}\psi_{\rho\upharpoonright\eta})$ and it is clear that this extension satisfies properties (1)-(4). Suppose that $i:(\Gamma,\psi)\to(\Gamma^*,\psi^*)$ is an embedding such that $(\Gamma^*,\psi^*)$ is a divisible $H$-asymptotic couple with asymptotic integration and there is a family $(\beta^*_{k,\eta})_{k\in\Z,\eta<\nu}$ in $\Psi^*$ satisfying the properties listed in (5). Then for each $\eta<\nu$ there is a unique extension of $i$ to an embedding $i_{\eta}:(\Gamma,\psi)\subseteq (\Gamma_{\rho\upharpoonright\eta},\psi_{\rho\upharpoonright\eta})\to (\Gamma^*,\psi^*)$ sending $\beta_{k,\eta_0}$ to $\beta_{k,\eta_0}^*$ for all $k\in\Z$ and $\eta_0<\eta$. Thus, it is clear that $i_{\nu} := \cup_{\eta<\nu}i_{\eta}:(\Gamma_{\rho},\psi_{\rho})\to (\Gamma^*,\psi^*)$ is an extension of $i$ sending $\beta_{k,\eta}$ to $\beta_{k,\eta}^*$ for all $k\in\Z$ and $\eta<\nu$. Uniqueness of $i_{\nu}$ follows from the observation that the restriction of $i_{\nu}$ to each $(\Gamma_{\rho\upharpoonright\eta},\psi_{\rho\upharpoonright\eta})$ is uniquely determined by the universal property that each $(\Gamma_{\rho\upharpoonright\eta},\psi_{\rho\upharpoonright\eta})$ enjoys (by induction).

Finally, (6) is immediate from the above construction.
\end{proof}

\medskip\noindent
In Figure~\ref{TransfiniteZzz}, we illustrate an instance of the construction done in Lemma~\ref{rhoZ} (over a model of $T_0$). Here we have the increasing function $\rho:4\to\sded(\Psi)$ where $\rho(0)<\rho(1)=\rho(2)<\rho(3)$. Since $\rho(1)=\rho(2)$, $(\beta_{k,1})$, the copy of $\Z$ corresponding to $\rho(1)$,  gets added to the same cut in $\Psi$ as $(\beta_{k,2})$ the copy of $\Z$ corresponding to $\rho(2)$. However, the construction ensures that $(\beta_{k,1})$ gets added entirely to the left of $(\beta_{k,2})$.

\begin{figure}[!htbp]
\centering
\caption{Example of Lemma~\ref{rhoZ} in action}
\label{TransfiniteZzz}
\resizebox{17cm}{!}{
\includegraphics{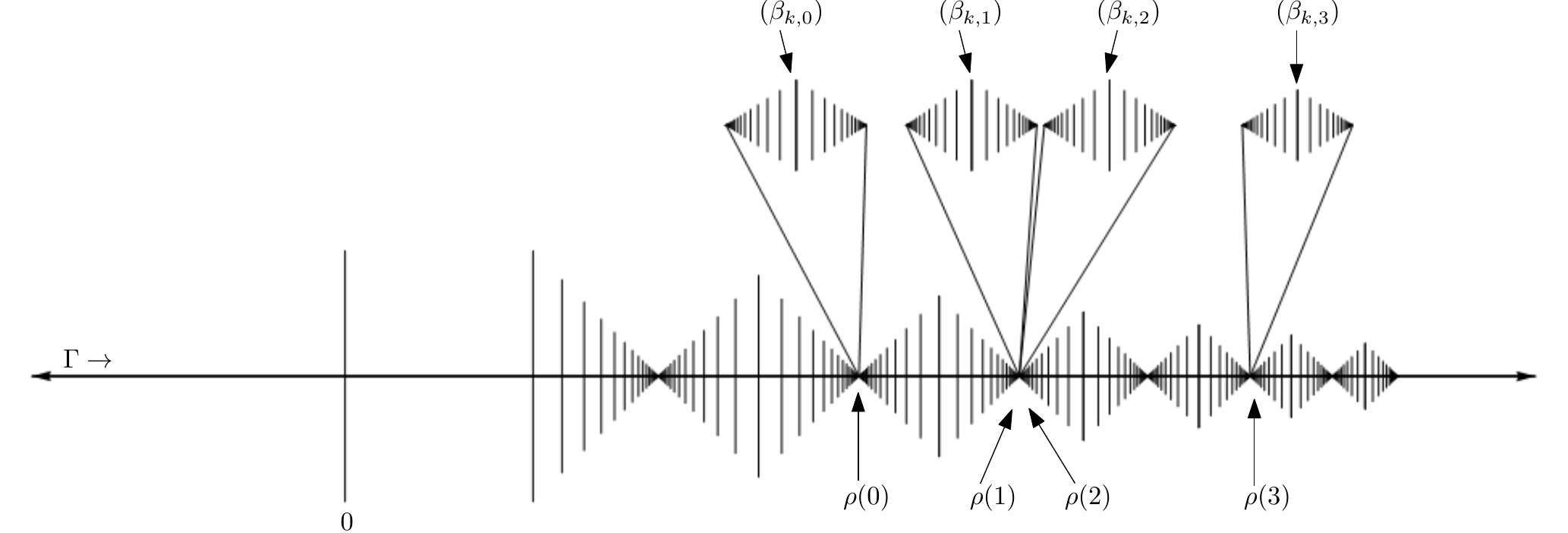}
}
\end{figure}

\medskip\noindent
For use in Section~\ref{sectionsimpleextensions}, we also recall here an important relationship between the functions $s$ and $\psi$:

\begin{prop}\cite[Corollary 3.5]{gehret}
\label{spsigap}
Let $(\Gamma^*,\psi^*)$ be an $H$-asymptotic couple with asymptotic integration that extends $(\Gamma,\psi)$. Suppose $\gamma^*\in\Psi^*$ is such that $\Psi<\gamma^*$. Then $s(\alpha) = \psi^*(\alpha-\gamma^*)$ for all $\alpha\in\Gamma$.
\end{prop}

\noindent
Note that Lemmas~\ref{Zinmiddle}  and~\ref{rhoZ} both give ways to construct such an extension $(\Gamma^*,\psi^*)$ as in Proposition~\ref{spsigap}. Proposition~\ref{spsigap} and Fact~\ref{valuationfact} also give the following useful way of computing values of the $\psi$-map:

\begin{cor}\cite[Lemma 3.4]{gehret}
\label{successorid}
For every $\alpha,\beta\in\Gamma$, if $s\alpha<s\beta$, then $\psi(\beta-\alpha) = s\alpha$.
\end{cor}

\medskip\noindent
As an application of Corollary~\ref{successorid} and Fact~\ref{valuationfact}, we obtain the following formulas for computing the $\psi$- and $s$-values in models of $T_0$:

\begin{lemma}\label{psisformulas}\cite[Lemma 6.4 and Corollary 6.5]{gehret}
Suppose $(\Gamma,\psi)\models T_0$. Let $n\geq 1, \alpha_1<\cdots <\alpha_n\in\Psi$, and let $\alpha = \sum_{j=1}^n q_j\alpha_j$ for $q_1,\ldots,q_n\in\Q^{\neq}$. Then
\begin{enumerate}
\item $\sum_{j=1}^n q_j=0\Longrightarrow \psi(\alpha_1) = s(\alpha_1)$,
\item $\sum_{j=1}^n q_j\neq 0 \Longrightarrow \psi(\alpha) = s0$,
\item $\sum_{j=1}^n q_j = 1\Longrightarrow s(\alpha) = s(\alpha_1)$,
\item $\sum_{j=1}^n q_j\neq 1 \Longrightarrow s(\alpha) = s0$.
\end{enumerate}
\end{lemma}

\medskip\noindent
Lemma~\ref{psisformulas} was useful in proving~\cite[Corollary 7.2]{gehret}: in models of $T_0$, the subset $\Psi$ of $\Gamma$ is stably embedded in $(\Gamma,\psi)$. We use it here in Section~\ref{examples} below.

\section{Simple Extensions}
\label{sectionsimpleextensions}

\noindent
For a model $(\Gamma,\psi)$ of $T_0$, we define the function $p:\Psi^{>s0}\to\Psi$ to be the inverse to the function $\gamma\mapsto s\gamma:\Psi\to\Psi^{>s0}$. We extend $p$ to a function $\Gamma_{\infty}\to\Gamma_{\infty}$ by setting $p(\alpha) := \infty$ for $\alpha\in\Gamma_{\infty}\setminus\Psi^{>s0}$.

\medskip\noindent
Next let $L_{\log} = L_{AC}\cup\{s,p,\delta_1,\delta_2,\delta_3,\ldots\}$ where $s$, $p$ and $\delta_n$ for $n\geq 1$ are unary function symbols. All models of $T_0$ are considered as $L_{\log}$-structures in the obvious way, again with $\infty$ as a default value, and with $\delta_n$ interpreted as division by $n$. 

\medskip\noindent
We let $T_{\log}$ be the $L_{\log}$-theory whose models are the models of $T_0$. By adding function symbols $s,p,\delta_1,\delta_2,\ldots$ we have guaranteed that $T_{\log}$ has a universal axiomatization, has quantifier elimination, is complete and is model complete; see Section 5 of~\cite{gehret}.

\medskip\noindent
\emph{For the rest of this section we let $\GAMMA = (\GAMMA,\psi,s,p,\ldots)$ be a monster model of $T_{\log}$.}
All other models considered will be small submodels of $\GAMMA$. In particular, we consider an arbitrary $\Gamma = (\Gamma,\psi,s,p,\ldots)$ of cardinality $\leq\kappa<\kappa(\M)$. The element $\alpha$ will range over $\GAMMA$ and we will assume $\alpha\not\in\Gamma$ to avoid some trivial cases. Note that the set $\Psi = \Psi_{\Gamma}$ will always contain the initial copy of $\N$ together with at most $\kappa$-many copies of $\Z$, whereas the set $\Psi_{\M}\setminus\Psi$ is the union of all copies of $\Z$ in $\Psi_{\M}$ that aren't part of $\Psi$.

\medskip\noindent
When considering simple extensions $\Gamma\langle\alpha\rangle$ of $\Gamma$ (in the language $L_{\log}$), it is useful to know whether the ordered abelian group $\Gamma\oplus\Q\alpha$ is already closed under the primitives $\psi$ and $s$. If it is not closed, then we want to know how badly $\Gamma\oplus\Q\alpha$ fails to be closed under $\psi$ and $s$. This motivates defining the following subsets of $\Psi_{\GAMMA}$:
\begin{align*}
\Q^{\neq}\alpha-\Gamma &:= \{q\alpha-\gamma:q\in\Q^{\neq}\text{ and }\gamma\in\Gamma\} \\
\psi(\Q^{\neq}\alpha-\Gamma) &:= \{\psi(q\alpha-\gamma):q\in\Q^{\neq}\text{ and }\gamma\in\Gamma\} \\
s(\Q^{\neq}\alpha-\Gamma) &:=  \{s(q\alpha-\gamma):q\in\Q^{\neq}\text{ and }\gamma\in\Gamma\} \\
T_{\Gamma}(\alpha) &:= \psi(\Q^{\neq}\alpha-\Gamma) \cup s(\Q^{\neq}\alpha-\Gamma).
\end{align*}
Note that $\psi(\Q^{\neq}\alpha-\Gamma) = \psi(\alpha-\Gamma) := \{\psi(\alpha-\gamma):\gamma\in\Gamma\}$ by (AC2). 

\medskip\noindent
Since $T_{\Gamma}(\alpha)$ is defined using the primitives $\psi$ and $s$, and $\alpha\not\in\Gamma$, it is clear that $T_{\Gamma}(\alpha)\subseteq\Psi_{\M}$. If $T_{\Gamma}(\alpha)\subseteq \Psi = \Psi_{\Gamma}$, then the ordered abelian group $\Gamma\oplus\Q\alpha$ is already closed under the primitives $\psi$ and $s$. However, if $T_{\Gamma}(\alpha)\setminus \Psi$ is nonempty, then $\Gamma\oplus\Q\alpha$ is not closed under $\psi$ and $s$ and then we are interested in the possibilities of the set $T_{\Gamma}(\alpha)\setminus \Psi$. 

\medskip\noindent
As we will show below in Corollary~\ref{tracetypes}, the set $T_{\Gamma}(\alpha)\setminus\Psi$ is either empty, or contains a single element in $\Psi_{\M}\setminus \Psi$. At any rate, since $T_{\Gamma}(\alpha)\subseteq\Gamma\langle\alpha\rangle$, all elements of $T_{\Gamma}(\alpha)\setminus\Psi$ must get added to $\Gamma$ in order to have any chance at closing off under $s$ and $\psi$.

\begin{remark}
\label{alsop} In fact, $T_{\Gamma}(\alpha)\setminus\Psi$ also measures the failure of $\Gamma\oplus\Q\alpha$ to be closed under $p$ in the following way: if $p(q\alpha-\gamma)\in\Psi_{\GAMMA}\setminus\Psi$, then $q\alpha-\gamma\in\Psi_{\GAMMA}\setminus\Psi$ and in particular, $s(q\alpha-\gamma)\in \Psi_{\GAMMA}\setminus\Psi$. For such a $q\alpha-\gamma$, $p(q\alpha-\gamma)$ and $s(q\alpha-\gamma)$ will be on the same copy of $\Z$ in $\Psi_{\GAMMA}\setminus\Psi$. Thus if $\Gamma\oplus\Q\alpha$ is not closed under $p$, then this failure is already recognized by the fact that $\Gamma\oplus\Q\alpha$ isn't closed under $s$.
\end{remark}

\medskip\noindent
In view of Proposition~\ref{spsigap} which relates the functions $\psi$ and $s$ through a translation by an external parameter, it may come as no surprise that $\psi(\Q^{\neq}\alpha-\Gamma)$ and $s(\Q^{\neq}\alpha-\Gamma)$ are very similar as the following two lemmas show:

\begin{lemma}
\label{downward}
Let $\Delta$ be either $\psi(\Q^{\neq}\alpha-\Gamma)$ or $s(\Q^{\neq}\alpha-\Gamma)$. Then for $\beta_0\in\GAMMA$, $\beta_1\in\Delta$ such that $\beta_0<\beta_1$, we have $\beta_0\in\Psi$ iff $\beta_0\in\Delta$. In particular, $\Delta\cap\Psi$ is a downward closed subset of $\Psi$ and $\Delta\setminus\Psi$ consists of at most one element $\beta$; furthermore, such $\beta$ realizes the cut $(\Delta\cap\Psi, \Psi\setminus\Delta)$ in $\Psi$.
\end{lemma}
\begin{proof}
First, consider the case that $\Delta = \psi(\Q^{\neq}\alpha-\Gamma) = \psi(\alpha-\Gamma)$ and let $\beta_0\in\M$ and $\beta_1\in\Delta$ be arbitrary such that $\beta_0<\beta_1$. Then $\beta_1=\psi(\alpha-\gamma_1)$ for some $\gamma_1\in \Gamma$.  First suppose that $\beta_0\in \Psi$. Then there is $\gamma_0\in\Gamma$ such that $\beta_0 = \psi(\gamma_0)<\psi(\alpha-\gamma_1)=\beta_1$. Note that
\[
\beta_0 = \psi(\gamma_0) = \psi(\gamma_0-(\alpha-\gamma_1)) = \psi(\alpha-(\gamma_0+\gamma_1))\in\Delta.
\]
Conversely, if $\beta_0\in\Delta$, then $\beta_0=\psi(\alpha-\gamma_0)$ for some $\gamma_0\in\Gamma$. It then follows from $\beta_0=\psi(\alpha-\gamma_0)<\psi(\alpha-\gamma_1)=\beta_1$ that
\[
\beta_0 = \psi(\alpha-\gamma_0) = \psi((\alpha-\gamma_0)-(\alpha-\gamma_1)) = \psi(\gamma_1-\gamma_0)\in\Psi.
\]


Next, consider the case that $\Delta = s(\Q^{\neq}\alpha-\Gamma)$ and let $\beta_0\in\M$ and $\beta_1\in\Delta$ be arbitrary such that $\beta_0<\beta_1$. Then $\beta_1 = s(q_1\alpha-\gamma_1)$ for some $q_1\in\Q^{\neq}$ and $\gamma_1\in\Gamma$. We will also take $\gamma^*\in\Psi_{\GAMMA}$ such that $\gamma^*>\Psi_{\Gamma\langle\alpha\rangle}$.  
First suppose that $\beta_0\in\Psi$. Then $\beta_0 = \psi(\gamma_0)$ for some $\gamma_0\in\Gamma$ and thus $\beta_0 = \psi(\gamma_0)<s(q_1\alpha-\gamma_1) = \beta_1$. Then by Proposition~\ref{spsigap},
\[
\beta_0 = \psi(\gamma_0) =  \min(s(q_1\alpha-\gamma_1),\psi(\gamma_0)) = \min(\psi(q_1\alpha-\gamma_1-\gamma^*),\psi(\gamma_0)) 
\]
\[
= \psi(q_1\alpha-\gamma_1-\gamma^* - \gamma_0) = s(q_1\alpha-(\gamma_1+\gamma_0)) \in \Delta.
\]
Conversely, if $\beta_0\in \Delta$, then $\beta_0 = s(q_0\alpha-\gamma_0)$ for some $q_0\in\Q^{\neq}$ and $\gamma_0\in\Gamma$. Then 
$\beta_0=s(q_0\alpha-\gamma_0)<s(q_1\alpha-\gamma_1) = \beta_1$, and it follows that
\[
\beta_0 = \psi(\alpha-q_0^{-1}\gamma_0-q_0^{-1}\gamma^*) < \psi(\alpha-q_1^{-1}\gamma_1-q_1^{-1}\gamma^*)
\]
and so
\[
\beta_0 = \psi(q_1^{-1}\gamma_1- q_0^{-1}\gamma_0 + (q_1^{-1}-q_0^{-1})\gamma^*).
\]
If $q_0=q_1$, then $\beta_0\in \Psi$. Otherwise,
\[
\beta_0 = \psi\left( -  \frac{q_1^{-1}}{q_1^{-1}-q_0^{-1}}\gamma_1+ \frac{q_0^{-1}}{q_1^{-1}-q_0^{-1}}\gamma_0 - \gamma^*  \right) = s\left(-  \frac{q_1^{-1}}{q_1^{-1}-q_0^{-1}}\gamma_1+ \frac{q_0^{-1}}{q_1^{-1}-q_0^{-1}}\gamma_0\right) \in\Psi. \qedhere
\]

\end{proof}

\begin{lemma}
\label{nearlyequal}
$s(\Q^{\neq}\alpha-\Gamma)\cap\Psi= \psi(\Q^{\neq}\alpha-\Gamma)\cap\Psi$. Furthermore, $s(\Q^{\neq}\alpha-\Gamma)\triangle\psi(\Q^{\neq}\alpha-\Gamma)$ consists of at most one element.
\end{lemma}
\begin{proof}
Suppose $\beta_0\in s(\Q^{\neq}\alpha-\Gamma)\cap\Psi$. Let $q\in\Q^{\neq}$ and $\gamma_0\in\Gamma$ be such that $\beta_0 = s(q\alpha-\gamma_0)$. Let $\gamma_1\in\Gamma$ be such that $s(\gamma_1)>\beta_0 = s(q\alpha-\gamma_0)$. Then Corollary~\ref{successorid} implies
\[
\psi(q\alpha-(\gamma_1+\gamma_0)) = \psi(\gamma_1 - (q\alpha-\gamma_0)) = s(q\alpha-\gamma_0) = \beta_0
\]
and so $\beta_0\in\psi(\Q^{\neq}\alpha-\Gamma)\cap\Psi$.

Next we consider two cases. First suppose $s(\Q^{\neq}\alpha-\Gamma)\cap\Psi$ is cofinal in $\Psi$. Since it is also downward closed in $\Psi$, it is necessarily the case that $\Psi = s(\Q^{\neq}\alpha-\Gamma)\cap\Psi \subseteq \psi(\Q^{\neq}\alpha-\Gamma)\cap\Psi\subseteq\Psi$, so we get equality throughout.

Otherwise, by Lemma~\ref{downward} we can take $\rho\in\Psi$ such that $s(\Q^{\neq}\alpha-\Gamma)<\rho$. Let $\gamma\in\Gamma$ be arbitrary such that $\psi(\alpha-\gamma)\in \psi(\Q^{\neq}\alpha-\Gamma)\cap\Psi$. Then by choice of $\rho\in \Psi$ we have
\[
\underbrace{s(\alpha-\gamma+\rho)}_{\in s(\Q^{\neq}\alpha-\Gamma)} <\rho < s(\rho).
\]
Thus by Corollary~\ref{successorid} we have
\[
\psi(\alpha-\gamma) = \psi((\alpha-\gamma+\rho)-\rho) = s(\alpha-\gamma+\rho)\in s(\Q^{\neq}\alpha-\Gamma)
\]
and we conclude that $ \psi(\Q^{\neq}\alpha-\Gamma)\cap\Psi\subseteq s(\Q^{\neq}\alpha-\Gamma)\cap\Psi$.

Finally, suppose $s(\Q^{\neq}\alpha-\Gamma)\setminus\Psi = \{\beta\}$ where $\beta\in\Psi_{\GAMMA}\setminus\Psi$. Then we will show that $\psi(\Q^{\neq}\alpha-\Gamma) \subseteq\Psi\cup\{ \beta\}$. 
Take $\gamma^{\ast}\in\Psi_{\GAMMA}$ such that $\gamma^{\ast}>\Psi_{\Gamma\langle\alpha\rangle}$ and take $q\in\Q^{\neq}$ and $\gamma\in\Gamma$ such that $\beta = s(q\alpha-\gamma) = \psi(q\alpha-\gamma-\gamma^{\ast})$. Let $\delta\in\Gamma$ be arbitrary. Note that
\begin{align*}
\psi(q\alpha-\gamma-\delta) &= \psi((q\alpha-\gamma-\gamma^{\ast})-(\delta-\gamma^{\ast})) \\
&\geq \min(\psi(q\alpha-\gamma-\gamma^{\ast}),\psi(\delta-\gamma^{\ast})) \\
&= \min(s(q\alpha-\gamma),s(\delta)) \\
&= \min(\beta,s(\delta)).
\end{align*}
But since $\beta\not\in\Psi$ and $s(\delta)\in\Psi$, we actually get $\psi(q\alpha-\gamma-\delta) = \min(\beta,s(\delta))$. Since $q\neq 0$ and as $\delta$ ranges over $\Gamma$, $\gamma+\delta$ will also range over $\Gamma$, together with (AC2) this argument shows that $\psi(\Q^{\neq}\alpha-\Gamma)\subseteq \Psi\cup\{\beta\}$.
\end{proof}

\medskip\noindent
It follows that $T_{\Gamma}(\alpha)$ occurs in only three different ways:

\begin{cor}
\label{tracetypes}
Exactly one of the following is true:
\begin{enumerate}
\item $T_{\Gamma}(\alpha) = [s0,\beta]_{\Psi} = \Psi^{\leq\beta}\subseteq \Psi$ for some $\beta\in\Psi$.
\item $T_{\Gamma}(\alpha) = B$ where $B\subseteq\Psi$ is nonempty, downward closed and is such that $s(B)\subseteq B$ (i.e., $\Psi\setminus B\in\sded(\Psi)$).
\item $T_{\Gamma}(\alpha) = B\cup\{\beta\}$ where $B\subseteq\Psi$ is nonempty, downward closed and is such that $s(B)\subseteq B$ and $\beta\in\Psi_{\GAMMA}\setminus\Psi$ and $B<\beta<(\Psi\setminus B)$.
\end{enumerate}
In particular, $|T_{\Gamma}(\alpha)\setminus\Psi|\leq 1$.
\end{cor}

\medskip\noindent
Note that if $T_{\Gamma}(\alpha)\subseteq\Psi$ for a particular $\Gamma$ and $\alpha\in\GAMMA$, then $\Gamma\oplus\Q\alpha$ as an ordered abelian subgroup of $\GAMMA$ is closed under the functions $\psi$ and $s$. In fact, it follows from Remark~\ref{alsop} that $\Gamma\oplus\Q\alpha$ is also closed under $p$. Thus $(\Gamma\oplus\Q\alpha,\psi)$ is an $L_{\log}$-substructure of $\GAMMA$ which extends $\Gamma$ and hence also is a model of $T_{\log}$ since $T_{\log}$ has a universal axiomatization. In this case, $\Gamma\langle\alpha\rangle = (\Gamma\oplus\Q\alpha,\psi)$.

\medskip\noindent
The following observation illustrates how the inductive step in Theorem~\ref{simpleextensions} below will work:

\begin{observation}
\label{increasingobservation}
Suppose that $\Gamma_0\subseteq\Gamma_1\subseteq \M$ are models of $T_{\log}$ and that $\alpha\in\M\setminus\Gamma_1$. Then $T_{\Gamma_0}(\alpha)\subseteq T_{\Gamma_1}(\alpha)$. In particular, if $T_{\Gamma_0}(\alpha) = B_0\cup\{\beta_0\}$ as in case (3) of Corollary~\ref{tracetypes} above, and if $\Gamma_1 = \Gamma_0\langle \beta_0\rangle = \Gamma_0+\sum_{n}\Q s^n\beta_0+\sum_n\Q p^n\beta_0$ also has the property that $T_{\Gamma_1}(\alpha) = B_1\cup\{\beta_1\}$ as in case (3) of Corollary~\ref{tracetypes}, then it must be the case that $\beta_0\in B_1$ and thus $s^n\beta_0<\beta_1$ for all $n$.
\end{observation}

\begin{thm}
\label{simpleextensions}
Let $\alpha\in\GAMMA$. Then $\Gamma\langle\alpha\rangle$ is isomorphic over $\Gamma$ to one of the following:
\begin{enumerate}
\item $\Gamma_{\rho}$ for some increasing $\rho:n\to\sded(\Psi)\setminus\{\Psi\}$ and some $n$,
\item $\Gamma_{\rho}\oplus\Q\alpha$ for some increasing $\rho:n\to\sded(\Psi)\setminus\{\Psi\}$ and some $n$,
\item $\Gamma_{\rho}\oplus\Q\alpha$ for some increasing $\rho:\omega\to\sded(\Psi)\setminus\{\Psi\}$.
\end{enumerate}
\end{thm}
\begin{proof}
We will recursively construct a sequence of extensions $\Gamma=:\Gamma_0\subseteq\Gamma_1\subseteq\Gamma_2\subseteq\cdots\subseteq\Gamma\langle\alpha\rangle$ of models of $T$  inside $\GAMMA$. This sequence will either be finite or have order type $\omega+1$ and the last element of the sequence will be $\Gamma\langle\alpha\rangle$.

We will inductively assume that each $\Gamma_n$ constructed so far is isomorphic to some $\Gamma_{\rho}$ for some increasing $\rho:n\to\sded(\Psi)\setminus\{\Psi\}$. This is true for $n=0$ since $\Gamma_0 = \Gamma = \Gamma_{\rho}$ for the empty increasing function $\rho:0\to\sded(\Psi)\setminus\{\Psi\}$. Given $\Gamma_n$ for $n<\omega$, if $\alpha\in\Gamma_n$, then we are done, i.e., $\Gamma\langle\alpha\rangle = \Gamma_n$ and so $\Gamma\langle\alpha\rangle\cong\Gamma_{\rho}$ for some increasing $\rho:n\to\sded(\Psi)\setminus\{\Psi\}$. 
Otherwise, consider the set $T_{\Gamma_n}(\alpha)$. If $T_{\Gamma_n}(\alpha)\subseteq\Psi_{\Gamma_n}$ then we set $\Gamma_{n+1}:=\Gamma_n\oplus\Q\alpha$ and we are done, i.e., $\Gamma\langle\alpha\rangle = \Gamma_{n+1}\cong\Gamma_{\rho}\oplus\Q\alpha$ for some increasing $\rho:n\to\sded(\Psi)\setminus\{\Psi\}$. 

Otherwise, we are in the case where $T_{\Gamma_n}(\alpha) = B\cup\{\beta\}$ where $B\subseteq\Psi_n$ is nonempty, downward closed and is such that $s(B)\subseteq B$ and $\beta\in\Psi_{\GAMMA}\setminus\Psi_n$ and $B<\beta<(\Psi_n\setminus B) = \Psi\setminus B$. In this case we set $\Gamma_{n+1}:=\Gamma_{n}\langle\beta\rangle$, i.e., we add to $\Gamma_n$ the element $\beta$, and with it, the entire copy of $\Z$ that $\beta$ lives on, so $\Gamma_{n+1} = \Gamma_n+\sum_n\Q p^n\beta+\sum_n\Q s^n\beta$. Thus $\Gamma_{n+1}\cong (\Gamma_n)_{(\Psi_n\setminus B)}$. By Observation~\ref{increasingobservation} we actually have $\Gamma_{n+1}\cong \Gamma_{\rho'}$ for some increasing $\rho':n+1\to\sded(\Psi)\setminus\{\Psi\}$. Now that we've constructed $\Gamma_{n+1}$, we keep going.

Note that we either terminate the construction at a finite $n$ or else $\bigcup_n\Gamma_n$ is isomorphic to $\Gamma_{\rho}$ inside $\Gamma\langle\alpha\rangle$ for some increasing $\rho:\omega\to\sded(\Psi)$, by Observation~\ref{increasingobservation}. In the latter case, we note that $\Gamma_{\omega}:=(\bigcup_n\Gamma_n)\oplus\Q\alpha$ (as an ordered abelian group) is automatically closed under $\psi$ and $s$ by construction and so we are done: $\Gamma\langle\alpha\rangle = \Gamma_{\omega}$ and so $\Gamma\langle\alpha\rangle\cong \Gamma_{\rho}\oplus\Q\alpha$ for some increasing $\rho:\omega\to\sded(\Psi)\setminus\{\Psi\}$.
\end{proof}

\section{Examples}
\label{examples}

\noindent
In this section, we give explicit examples of extensions of models of $T_{\log}$ which realize each type of simple extension in Theorem~\ref{simpleextensions}.

\medskip\noindent
First, we recall the useful notion of \emph{pseudocauchy sequences} and \emph{pseudolimits} from valuation theory, given here only in the special context of asymptotic couples with valuation map $\psi$:

\begin{definition}
\label{pcdef}
Let $(\Gamma,\psi)$ be an asymptotic couple and $\nu\neq 0$ a limit ordinal. A sequence $(\alpha_{\rho})_{\rho<\nu}$ in $\Gamma$ is a \textbf{pseudocauchy sequence}, or \textbf{pc-sequence}, in $(\Gamma,\psi)$ if for some index $\rho_0<\nu$ we have
\[
\rho_0<\rho<\sigma<\tau<\nu \Longrightarrow \psi(\alpha_{\rho}-\alpha_{\sigma})<\psi(\alpha_{\sigma}-\alpha_{\tau}).
\]
For $\alpha\in\Gamma$, the sequence $(\alpha_{\rho})_{\rho<\nu}$ in $\Gamma$ is said to \textbf{pseudoconverge to $\alpha$}, and $\alpha$ is a \textbf{pseudolimit of $(\alpha_{\rho})_{\rho<\nu}$} if for some index $\rho_0<\nu$ we have
\[
\rho_0<\rho<\sigma<\nu \Longrightarrow \psi(\alpha-\alpha_{\rho})<\psi(\alpha-\alpha_{\sigma}).
\]
\end{definition}

\noindent
The basic connection between pc-sequences and model theory is the following:

\begin{lemma}
\label{pclimit}
Let $(\Gamma,\psi)$ be an asymptotic couple, $\nu\neq 0$ a limit ordinal, and $(\alpha_{\rho})_{\rho<\nu}$ a pc-sequence in $\Gamma$. Then there is an elementary extension $(\Gamma^*,\psi^*)$ of $(\Gamma,\psi)$ and an element $\alpha\in\Gamma^*$ such that $(\alpha_{\rho})_{\rho<\nu}$ pseudoconverges to $\alpha$.
\end{lemma}
\begin{proof}
Suppose $(\alpha_{\rho})_{\rho<\nu}$ is a pc-sequence in $\Gamma$, with $\nu\neq 0$ a limit ordinal. Let $\rho_0<\nu$ be as in Definition~\ref{pcdef}. Consider the partial type given by all formulas of the form
\[
\psi(x-\alpha_{\rho})<\psi(x-\alpha_{\sigma})
\]
for $\rho_0<\rho<\sigma$. Since every finite subset of this type is realized in $(\Gamma,\psi)$, this type will be realized by an element $\alpha$ in an elementary extension $(\Gamma^*,\psi^*)$ of $(\Gamma,\psi)$. It easily follows that $\alpha$ is a pseudolimit of the sequence $(\alpha_{\rho})_{\rho<\nu}$.
\end{proof}

\subsection{Example 1} Consider the $L_{\log}$-substructure $(\Gamma_{\log}^{\Q},\psi)$ of $(\Gamma_{\log},\psi)$ with underlying group $\Gamma_{\log}^{\Q} := \sum_n\Q e_n$. In~\cite{gehret} we showed that $(\Gamma_{\log}^{\Q},\psi)\models T_{\log}$, and in fact, $(\Gamma_{\log}^{\Q},\psi)$ is a prime model of $T_{\log}$. Let $\alpha$ be the element
\[
\alpha := \sqrt{2}e_2 = (0,0,\sqrt{2},0,\ldots)\in \Gamma_{\log}\setminus \Gamma_{\log}^{\Q}.
\]
An arbitrary element of $\Q^{\neq}\alpha-\Gamma^{\Q}_{\log}$ looks like
\[
(q_0,q_1,\underbrace{q_2+q\sqrt{2}}_{\neq 0,1},q_3,\ldots)
\]
where $q\in\Q^{\neq}$ and $q_n\in\Q$, where $q_n=0$ for all but finitely many $n$. Since the third entry $q_2+q\sqrt{2}$ can never be $0$ or $1$,
a computation using Example~\ref{functionformulas} shows that
\[
\psi(\Q^{\neq}\alpha-\Gamma_{\log}^{\Q}) = s(\Q^{\neq}\alpha-\Gamma_{\log}^{\Q}) = \{s0,s^20,s^30\}
\]
and thus
\[
T_{(\Gamma_{\log}^{\Q},\psi)}(\alpha) = \{s0, s^20, s^30\} = [e_0, e_0+e_1+e_2]_{\Psi_{(\Gamma_{\log}^{\Q},\psi)}} \subseteq \Psi_{(\Gamma_{\log}^{\Q},\psi)}.
\]
Therefore
\[
(\Gamma_{\log}^{\Q},\psi)\langle\alpha\rangle = (\Gamma_{\log}^{\Q}\oplus\Q\alpha,\psi)
\]
where the direct sum is taken inside $\Gamma_{\log}$ and $\psi$ is the restriction of the $\psi$-map of $(\Gamma_{\log},\psi)$. This is an example of (2) from Theorem~\ref{simpleextensions} and (1) from Corollary~\ref{tracetypes}.

\subsection{Example 2} 
The idea for this example is to adjoin the vector
\[
\left(1,\frac{1}{2},\frac{1}{3},\frac{1}{4},\frac{1}{5},\ldots\right)
\]
to the asymptotic couple $(\Gamma_{\log},\psi)$. This can be made precise using the notions of pc-sequences and pseudolimits as follows:

Consider the sequence $(\alpha_N)_{N<\omega}:=(\sum_{i=0}^N(1+i)^{-1}e_i)_{N<\omega}$ in $(\Gamma_{\log},\psi)$. If $N_0<N_1<\omega$, then
\begin{align*}
\alpha_{N_0}-\alpha_{N_1} &= -\textstyle\sum_{i=N_0+1}^{N_1}(1+i)^{-1}e_i,\quad\text{and thus:} \\
\stepcounter{equation}\tag{\theequation}\label{pcseqcomp} \psi(\alpha_{N_0}-\alpha_{N_1}) &= \textstyle\sum_{i=0}^{N_0+1}e_i = s^{N_0+1}0,\quad\text{for all $N_0<N_1<\omega$}
\end{align*}
This shows that $(\alpha_N)_{N<\omega}$ is a pc-sequence in $(\Gamma_{\log},\psi)$. By Lemma~\ref{pclimit}, we get an elementary extension $(\Gamma^*,\psi^*)$ of $(\Gamma_{\log},\psi)$ and an element $\alpha\in\Gamma^*$ such that $\alpha$ is a pseudolimit of $(\alpha_N)_{N<\omega}$. 
In some sense $\alpha$ can be thought of as the vector above, especially when it comes to doing calculations.
It follows from (\ref{pcseqcomp}) and the definition of pseudolimit that
\[
\stepcounter{equation}\tag{\theequation}\label{pclimitcomp} \psi(\alpha-\alpha_N) = s^{N+1}0,\quad\text{for all $N<\omega$.}
\]

Let $\gamma = \sum_nq_ne_n\in\Gamma_{\log}$ be arbitrary, where $q_n\in\Q$ for all $n$. Then take the unique $N<\omega$ such that $q_n = (1+n)^{-1}$ iff $n<N$. Next let $M<\omega$ be arbitrary and note that
\begin{align*}
\psi(\gamma - \alpha_{N+M}) &= \psi\left(\textstyle\sum_n q_ne_n - \sum_{n=0}^{M+N}(1+n)^{-1}e_n\right)\\
 &= \psi\left(\textstyle\sum_{n\geq N}q_ne_n - \textstyle\sum_{n=N}^{N+M}(1+n)^{-1}e_n\right) \\
 &= \psi\bigg(\underbrace{(q_N-(1+N)^{-1})}_{\neq 0}e_N + \textstyle\sum_{n>N}q_n^*e_n\bigg) \quad\text{(for some $q_n^*\in\Q$)} \\
 &= \textstyle\sum_{n=0}^Ne_n = s^N0.
\end{align*}
In light of (\ref{pclimitcomp}), this computation shows that $\alpha\in\Gamma^*\setminus\Gamma_{\log}$. Using Fact~\ref{valuationfact} and the definition of pseudolimit, the above computation also shows that
\[
\psi(\Q^{\neq}\alpha-\Gamma_{\log}) = \Psi_{\Gamma_{\log}}.
\]

To compute $s(\Q^{\neq}\alpha-\Gamma_{\log})$, let $q\in\Q^{\neq}$ and $\gamma = \sum_nq_ne_n\in\Gamma_{\log}$ be arbitrary. Take the unique $N<\omega$ such that $q_n = q(1+n)^{-1}-1$ iff $n<N$. Then we have
\[
q\alpha_{N+1}-\gamma = e_0+\cdots+e_{N-1} + \underbrace{(q(1+N)^{-1}-q_N)}_{\neq 1}e_N + \sum_{n>N}q_n^*e_n\quad\text{(for some $q_n^*\in\Q$)}
\]
and thus $s(q\alpha_{N+1} - \gamma) = s^N0$. Furthermore, (\ref{pclimitcomp}) implies that
\[
[q\alpha-q\alpha_{N+1}]<[e_N].
\]
Thus, with $\tilde{q}:= |1-q(1+N)^{-1}+q_N|/2\in\Q^{>}$, we have that
\[
q\alpha_{N+1}-\gamma - \tilde{q}e_N < q\alpha-\gamma = (q\alpha-q\alpha_{N+1}) + (q\alpha_{N+1}-\gamma) < q\alpha_{N+1}-\gamma + \tilde{q}e_N,
\]
with all three quantities contained either entirely within $((\Gamma^*)^{<})'$ or entirely within $((\Gamma^*)^{>})'$. Thus by (4) and (5) of Lemma~\ref{functionproperties}, it follows that $s(q\alpha - \gamma) = s^N0$. This computation shows that
\[
s(\Q^{\neq}\alpha-\Gamma_{\log}) = \Psi_{\Gamma_{\log}}.
\]

We conclude that
\[
T_{(\Gamma_{\log},\psi)}(\alpha) = \Psi_{\Gamma_{\log}}
\]
and so
\[
(\Gamma_{\log},\psi)\langle\alpha\rangle = (\Gamma_{\log}\oplus \Q\alpha,\psi)
\]
where the direct sum is being taken in $(\Gamma^*,\psi^*)$ and $\psi$ is the restriction of the $\psi$-map of $(\Gamma^*,\psi^*)$. This is an example of (2) from Theorem~\ref{simpleextensions} and (2) from Corollary~\ref{tracetypes}.

\subsection{Example 3} In this example, we let $(\Gamma,\psi)$ be an arbitrary model of $T_{\log}$ and we fix an extension $(\Gamma_{\rho},\psi_{\rho})$ for some increasing $\rho:n\to\sded(\Psi)$ for some $n\geq 1$. Consider an element $\alpha\in \Gamma_{\rho}$ such that
\[
\alpha := \gamma + \sum_{j=0}^{n-1}\alpha_j
\]
where $\gamma\in\Gamma$ and $\alpha_j\in (\operatorname{span}_{\Q}(\beta_{k,j})_{k\in\Z})^{\neq}$, i.e., each $\alpha_j$ is constructed from a nontrivial linear combination of $\beta_{k,j}$'s from the $j$th copy of $\Z$ that was added to $\Gamma$ in $\Gamma_{\rho}$. We will show that $\alpha$ has the property that $\Gamma\langle\alpha\rangle = \Gamma_{\rho}$, and so it is in some sense a ``primitive element'' for the extension $\Gamma_{\rho}$ of $\Gamma$.

First, since $\Gamma\langle\alpha\rangle = \Gamma\langle \alpha-\gamma\rangle$, we may replace $\alpha$ with $\alpha-\gamma$. Thus $\alpha = \sum_{j=0}^{n-1}\alpha_j$. By the $\Q$-linear independence of the $(\beta_{k,j})_{k\in\Z, j<n}$ (see Lemma 6.8 of~\cite{gehret}), we may uniquely write $\alpha = \sum_{l=0}^Nq_l\beta_l$ for some $N>0$,  with $q_0,\ldots,q_N\in\Q^{\neq}$ and $(\beta_l)_{l\leq N}\subseteq (\beta_{k,j})_{k\in\Z, j<n}$ are such that $\beta_0<\cdots<\beta_{N}$.

Next, if $\sum_{l=0}^Nq_l =0$, then $\psi(\alpha) = s\beta_0\in \Gamma\langle\alpha\rangle$, otherwise $s((\sum_{l=0}^Nq_l)^{-1}\alpha) = s\beta_0$ (by Lemma~\ref{psisformulas}). Thus $(s^k\beta_0)_{k\in\Z}\subseteq\Gamma\langle\alpha\rangle$ and $\alpha-q_0\beta_0 = \sum_{l=1}^Nq_l\beta_l\in \Gamma\langle\alpha\rangle$. In this way, we have ``stripped off'' the least $\beta_{k,j}$ in $\alpha$ and we have recovered the first copy of $\Z$ in the construction of $\Gamma_{\rho}$. Continuing in this manner we can recover all the other copies of $\Z$.

It is also clear that all such ``primitive elements'' of $\Gamma_{\rho}$ must take this form. This simple extension is an example of (1) in Theorem~\ref{simpleextensions}.

%

\subsection{Example 4}
Finally we give an example of a simple extension of type (3) from Theorem~\ref{simpleextensions}. Let $(\Gamma,\psi)$ be an arbitrary model of $T_{\log}$ and we fix an extension $(\Gamma_{\rho},\psi_{\rho})$ for some increasing $\rho:\omega\to \sded(\Psi)$ inside $\M$. Let $(\beta_{k,j})_{k\in\Z,j<\omega}$ be the elements from the copies of $\Z$'s that were added to $\Gamma$ in $\Gamma_{\rho}$.

Next define the element $\alpha_n:=\sum_{j=0}^n \beta_{1,j}-\beta_{0,j}\in \Gamma_{\rho\upharpoonright (n+1)}\subseteq\Gamma_{\rho}\subseteq\M$. Note that from Example 3 above we have $\Gamma\langle\alpha_n\rangle = \Gamma_{\rho\upharpoonright(n+1)}$. Also note that by Lemma~\ref{psisformulas}, (1), we have that
\[
\stepcounter{equation}\tag{\theequation}\label{pcseqcomp2}\psi(\alpha_n-\alpha_m) = \psi\left(\textstyle\sum_{j=m+1}^n\beta_{1,j}-\beta_{0,j}\right) = s(\beta_{0,m+1}) = \beta_{1,m+1},\quad\text{for all $m<n<\omega$,}
\]
and so the sequence $(\alpha_n)_{n<\omega}$ is a pc-sequence. By saturation of $\M$, we can take an element $\alpha$ that is a pseudolimit of $(\alpha_n)$.

We claim that $\Gamma\langle\alpha\rangle$ is of the form $\Gamma_{\rho}\oplus\Q\alpha$. First, note that by (\ref{pcseqcomp2}) and the definition of pseudolimit,  it follows that
\[
\stepcounter{equation}\tag{\theequation}\label{pclimitcomp2} \psi(\alpha-\alpha_n) = \beta_{1,n+1},\quad\text{for all $n<\omega$.}
\]
Thus by Fact~\ref{valuationfact}, (\ref{pclimitcomp2}), and Lemma~\ref{psisformulas}, (1), we get
\[
\psi(\alpha) = \psi((\alpha-\alpha_0)+\alpha_0) = \min(\psi(\alpha-\alpha_0),\psi(\alpha_0)) =  \min(\beta_{1,1}, \beta_{1,0}) = \beta_{1,0}.
\]
From this it is clear that in fact $\alpha_0 = \beta_{1,0} - \beta_{0,0} = \beta_{1,0}-p\beta_{1,0}\in\Gamma_{\rho\upharpoonright1}\subseteq\Gamma\langle\alpha\rangle$. In general, if we show that $\alpha_0,\ldots,\alpha_m\in \Gamma_{\rho\upharpoonright (m+1)}\subseteq\Gamma\langle\alpha\rangle$, then we may consider the pc-sequence $(\alpha_n - \sum_{j=0}^m\alpha_m)_{n\geq m+1}$ which pseudoconverges to $\alpha - \sum_{j=0}^m\alpha_m$ in $\Gamma\langle\alpha\rangle$. Then we can recover $\beta_{1,m+1}$ and thus also $\alpha_{m+1}$ similar to above by computing $\psi(\alpha - \sum_{j=0}^m\alpha_m)$. 

Thus we have shown $\Gamma_{\rho}\subseteq\Gamma\langle\alpha\rangle$, from which it follows from the proof of Theorem~\ref{simpleextensions} that in fact $\Gamma\langle\alpha\rangle = \Gamma_{\rho}\oplus\Q\alpha$.

\section{Counting Types in $T_{\log}$}
\label{countingtypes}
\noindent
In this section, we derive a consequence of Theorem~\ref{simpleextensions} necessary for proving NIP for $T_{\log}$ in Section~\ref{NIPsettheory} below:
\begin{cor}
\label{countingtypescor}
If $(\Gamma,\psi)\models T_{\log}$, then
$
|S^1(\Gamma)| \leq \ded(|\Gamma|)^{\aleph_0}.
$
\end{cor}
\medskip\noindent
Under the assumptions of Section~\ref{sectionsimpleextensions}, it follows from the quantifier elimination for $T_{\log}$ that two elements $\alpha,\beta\in\M\setminus\Gamma$ have the same type over $\Gamma$ iff $\alpha$ and $\beta$ have the same isomorphism type over $\Gamma$, i.e., iff there is an isomorphism $\Gamma\langle\alpha\rangle\cong\Gamma\langle\beta\rangle$ over $\Gamma$ which sends $\alpha$ to $\beta$. This is how Corollary~\ref{countingtypescor} will follow from Theorem~\ref{simpleextensions}. However first we must be aware of the following:

\begin{tournantdangereux}
\label{tournantdangereux}
Suppose $\alpha,\beta\in\M\setminus\Gamma$ have the property that $\Gamma\langle\alpha\rangle = \Gamma\oplus\Q\alpha$ and $\Gamma\langle\beta\rangle = \Gamma\oplus\Q\beta$, which is a special case of (2) from Theorem~\ref{simpleextensions}. In this simplest of cases, it may be tempting to conclude that $\alpha$ and $\beta$ realize the same type over $\Gamma$ \emph{if and only if $\alpha$ and $\beta$ realize the same cut over $\Gamma$}. However, this is \emph{not} true in general. Consider the following scenario: Let $\delta,s\delta\in\Psi = \Psi_{\Gamma\langle\alpha\rangle} = \Psi_{\Gamma\langle\beta\rangle}$ be two adjacent members of the common $\Psi$-set. Consider the following sets of archimedean classes of $\Gamma$:
\[
C_0 := \{[\gamma]: \gamma\in\Gamma\text{ and } \psi(\gamma) = s\delta\}< C_1 := \{[\gamma]: \gamma\in\Gamma\text{ and } \psi(\gamma) = \delta\}.
\]
It could be the case that both $\alpha,\beta>0$ and $C_0<[\alpha],[\beta]<C_1$, which would guarantee that they realize the same cut over $\Gamma$. However, its possible that $\psi(\alpha) = \delta$ whereas $\psi(\beta) = s\delta$ and in this case $\alpha$ and $\beta$ wouldn't realize the same type over $\Gamma$. To account for this phenomenon, we need to take a small detour.
\end{tournantdangereux}

\subsection{Two More Embedding Lemmas: A Detour}
\noindent
\emph{In this subsection $(\Gamma,\psi)$ is a divisible $H$-asymptotic couple.} Here we recall two additional embedding lemmas for $H$-asymptotic couples which will help us deal with the issue raised in~\ref{tournantdangereux} above. The first is~\cite[Lemma 9.8.1]{mt}:

\begin{lemma}
\label{archclassesembedding}
Let $i:\Gamma\to G$ be an embedding of ordered abelian groups inducing a bijection $[\Gamma]\to [G]$. Then there is a unique function $\psi_G:G^{\neq}\to G$ such that $(G,\psi_G)$ is an $H$-asymptotic couple and $i:(\Gamma,\psi)\to (G,\psi_G)$ is an embedding.
\end{lemma}
\begin{proof}
The unique $\psi_G:G^{\neq}\to G$ is defined by $\psi_G(g):= i(\psi(\gamma))$ for $g\in G^{\neq}$ and $\gamma\in\Gamma^{\neq}$ with $[g] = [i(\gamma)]$.
\end{proof}

\begin{cor}
\label{typedeterminedbycut}
Suppose $(\Gamma\oplus\Q\alpha,\psi^{\alpha})$ and $(\Gamma\oplus\Q\beta,\psi^{\beta})$ are two $H$-asymptotic couple extensions of $(\Gamma,\psi)$ such that
\begin{enumerate}
\item $[\Gamma\oplus\Q\alpha] = [\Gamma]$, and
\item $\alpha$ and $\beta$ realize the same cut over $\Gamma$.
\end{enumerate}
Then the isomorphism $i:\Gamma\oplus\Q\alpha\to\Gamma\oplus\Q\beta$ of ordered abelian groups over $\Gamma$ which sends $\alpha$ to $\beta$ is also an isomorphism of $i:(\Gamma\oplus\Q\alpha,\psi^{\alpha})\to (\Gamma\oplus\Q\beta,\psi^{\beta})$ of asymptotic couples over $(\Gamma,\psi)$.
\end{cor}
\begin{proof}
By (1) we have that $\psi^{\alpha}((\Gamma\oplus\Q\alpha)^{\neq}) = \Psi$ and by (2) that $[\Gamma\oplus\Q\beta] = [\Gamma]$. Given $\gamma_0+q\alpha\in\Gamma\oplus\Q\alpha^{\neq}$, let $\gamma_1\in\Gamma^{\neq}$ be such that $[\gamma_0+q\alpha] = [\gamma_1]$. It follows from condition (2) that $[\gamma_1] = [\gamma_0+q\beta]$. Thus $i(\psi^{\alpha}(\gamma_0+i\alpha)) = \psi^{\alpha}(\gamma_0+i\alpha) = \psi(\gamma_1) = \psi^{\beta}(\gamma_0+q\beta) = \psi^{\beta}(i(\gamma_0+q\alpha))$, using Lemma~\ref{archclassesembedding} for $(\Gamma\oplus\Q\beta,\psi^{\beta})$.
\end{proof}

\medskip\noindent
The second embedding lemma is a divisible variant of~\cite[Lemma 9.8.7]{mt}:
\begin{lemma}
\label{cutembeddinglemma}
Let $(C_0,C_1)$ be a cut in $[\Gamma^{\neq}]$ and let $\beta\in \Gamma$ be such that $\beta<(\Gamma^{>})', \gamma^{\dagger}\leq\beta$ for all $\gamma\in\Gamma^{\neq}$ with $[\gamma]\in C_1$, and $\beta\leq\delta^{\dagger}$ for all $\delta\in\Gamma^{\neq}$ with $[\delta]\in C_0$. Then there exists an H-asymptotic couple $(\Gamma\oplus\Q\alpha,\psi^{\alpha})$ extending $(\Gamma,\psi)$, with $\alpha>0$, such that:
\begin{enumerate}
\item $[\alpha]$ realizes the cut $(C_0,C_1)$ in $[\Gamma^{\neq}]$, and $\psi^{\alpha}(\alpha) = \beta$;
\item given any embedding $i$ of $(\Gamma,\psi)$ into a divisible $H$-asymptotic couple $(\Gamma_1,\psi_1)$ and any element $\alpha_1\in\Gamma_1^{>}$ such that $[\alpha_1]$ realizes the cut $(\{[i(\delta)]:[\delta]\in C_0\}, \{[i(\delta)]:[\delta]\in C_1\})$ in $[i(\Gamma^{\neq})]$ and $\psi_1(\alpha_1) = \beta$, there is a unique extension of $i$ to an embedding $j:(\Gamma\oplus\Q\alpha,\psi^{\alpha})\to (\Gamma_1,\psi_1)$ with $j(\alpha) = \alpha_1$.
\end{enumerate}
\end{lemma}

\begin{cor}
\label{typedeterminedbyseveralthings}
Suppose $(\Gamma\oplus\Q\alpha,\psi^{\alpha})$ and $(\Gamma\oplus\Q\beta,\psi^{\beta})$ are two $H$-asymptotic couple extensions of $(\Gamma,\psi)$ such that: 
\begin{enumerate}
\item $\psi^{\alpha}((\Gamma\oplus\Q\alpha)^{\neq}) = \Psi = \psi^{\beta}((\Gamma\oplus\Q\beta)^{\neq})$,
\item $\alpha>0$ and $\beta>0$,
\item $\psi^{\alpha}(\alpha) = \psi^{\beta}(\beta)$, and
\item $[\alpha]\not\in[\Gamma]$, $[\beta]\not\in[\Gamma]$, and $[\alpha]$ and $[\beta]$ realize the same cut over $[\Gamma]$; 
\end{enumerate}
then necessarily $\alpha$ and $\beta$ realize the same cut over $\Gamma$ and the isomorphism $i:\Gamma\oplus\Q\alpha\to\Gamma\oplus\Q\beta$ of ordered abelian groups over $\Gamma$ which sends $\alpha$ to $\beta$ is also an isomorphism of $i:(\Gamma\oplus\Q\alpha,\psi^{\alpha})\to (\Gamma\oplus\Q\beta,\psi^{\beta})$ of asymptotic couples over $(\Gamma,\psi)$.
\end{cor}

\subsection{Back to Counting Types} \emph{For the rest of this section $\M$ will be a monster model of $T_{\log}$ and $\Gamma$ will be a small submodel of $\M$ of size $\kappa$.} As a warmup to proving Corollary~\ref{countingtypescor}, we first prove the following:

\begin{lemma}
\label{simplecountingtypes}
There are at most $\ded(\kappa)$-many types of the form $\tp(\alpha|\Gamma)$ where $\alpha\in\M\setminus\Gamma$ has the property that $\Gamma\langle\alpha\rangle = \Gamma\oplus\Q\alpha$ inside $\M$.
\end{lemma}
\begin{proof}
We have to count the isomorphism types of elements $\alpha\in\M\setminus\Gamma$ that have the property that $\Gamma\langle\alpha\rangle = \Gamma\oplus\Q\alpha$. Let $\alpha\in\M\setminus\Gamma$ have this property. There are two cases to consider:

(Case 1) $[\Gamma\oplus\Q\alpha] = [\Gamma]$. In this case the isomorphism type of $\alpha$ over $\Gamma$ is determined completely by it's cut over $\Gamma$ by Corollary~\ref{typedeterminedbycut}. Thus there are at most $\ded(\kappa)$-many types that fall into this case.

(Case 2) $[\Gamma\oplus\Q\alpha]\neq[\Gamma]$. In this case, there will be some $\gamma\in\Gamma$, $q\in\Q^{\neq}$ such that $\gamma+q\alpha>0$ and $[\gamma+q\alpha]\not\in[\Gamma]$. In this case, the isomorphism type of $\alpha$ over $\Gamma$ is completely determined by this choice of $\gamma\in\Gamma$, $q\in\Q^{\neq}$, the cut that $[\gamma+q\alpha]$ realizes in $[\Gamma]$ and the element $\delta\in\Psi$ such that $\psi(\gamma+q\alpha) = \delta$, by Corollary~\ref{typedeterminedbyseveralthings}. Thus there are at most $\kappa\cdot\aleph_0\cdot\ded(\kappa)\cdot\kappa = \ded(\kappa)$-many types that fall into this case.
\end{proof}

\begin{proof}[Proof of Corollary~\ref{countingtypescor}]
Let $\alpha\in\M\setminus\Gamma$. Then by Theorem~\ref{simpleextensions}, we have three cases:

(Case 1) $\Gamma\langle\alpha\rangle\cong \Gamma_{\rho}$ for some increasing $\rho:n\to\sded(\Psi)$, for some $n$. In this case, the isomorphism type of $\alpha$ over $\Gamma$ is completely determined by the map $\rho$ and the specific element of $\Gamma_{\rho}$ which maps to $\alpha$. Since $|\Gamma_{\rho}| = |\Gamma|$, for each $n$ this gives $\ded(\kappa)^n\cdot\kappa = \ded(\kappa)$-many isomorphism types over $\Gamma$. In total, Case 1 gives $\sum_{n<\omega}\ded(\kappa) = \ded(\kappa)$-many types.

(Case 2) $\Gamma\langle\alpha\rangle\cong\Gamma_{\rho}\oplus\Q\alpha$ for some increasing $\rho:n\to\sded(\Psi)$, for some $n$. In this case, the isomorphism type of $\alpha$ over $\Gamma$ is determined by the map $\rho$ and then the type of $\alpha$ over the image of $\Gamma_{\rho}$ in $\M$. By Lemma~\ref{simplecountingtypes}, Case 2 gives $\sum_{n<\omega}\ded(\kappa)^n\cdot\ded(\kappa) = \ded(\kappa)$-many types.

(Case 3) $\Gamma\langle\alpha\rangle\cong\Gamma_{\rho}\oplus\Q\alpha$ for some increasing $\rho:\omega\to\sded(\Psi)$. In this case, the isomorphism type of $\alpha$ over $\Gamma$ is also determined by the map $\rho$ and then the type of $\alpha$ over the image of $\Gamma_{\rho}$ in $\M$. By Lemma~\ref{simplecountingtypes}, Case 3 gives $\ded(\kappa)^{\aleph_0}\cdot \ded(\kappa) = \ded(\kappa)^{\aleph_0}$-many types.
\end{proof}

\section{NIP}
\label{NIPsettheory}

\noindent
In this section we derive the main result of this paper as an immediate consequence of Corollary~\ref{countingtypescor}:

\begin{thm}
\label{TloghasNIP}
$T_{\log}$ and $T_0$ have NIP.
\end{thm}

\medskip\noindent
\emph{For the rest of this section, $T$ is an arbitrary first-order theory with monster model $\M$.}

\begin{definition}
Let $R\subseteq \M^{m+n} = \M^m\times \M^n$ be a definable relation. We say that $R$, and any $L_{\M}$-formula $\phi(x,y)$ that defines $R$, has the \textbf{independence property} (or \textbf{IP}) if there are $(a_i)_{i\in\N}\subseteq \M^m$ and $(b_I)_{I\subseteq \N}\subseteq\M^n$  such that
\[
R(a_i,b_I)\Longleftrightarrow i\in I,\quad \text{for all $i\in\N$ and $I\subseteq\N$}.
\]
Otherwise we say that $R$, and any $L_{\M}$-formula $\phi(x,y)$ that defines $R$, does not have the independence property (or \textbf{has NIP}). 

\medskip\noindent
We say that $T$ \textbf{has NIP} if every definable relation $R\subseteq \M^{m+n}$ for every $m,n$ has NIP.
\end{definition}

\begin{definition}
Define the \textbf{stability function of $T$} to be the function
\[
g_T(\kappa) = \sup_{M\models T, |M| = \kappa}\left|\bigcup_{n<\omega}S^n(M)\right| = \sup_{M\models T, |M| = \kappa}\left|S^1(M)\right|.
\]
\end{definition}

\medskip\noindent
The main result concerning NIP and the function $g_T(\kappa)$ is the following:

\begin{prop}
\label{stabfunNIP}
If $T$ has NIP, then
\[
g_T(\kappa)\leq \ded(\kappa)^{|T|} \quad \text{for all $\kappa$},
\]
and if $T$ has the independence property, then
\[
g_T(\kappa) = 2^{\kappa}\quad \text{for all $\kappa$.}
\]
\end{prop}
\noindent
Proposition~\ref{stabfunNIP} is a global form of~\cite[Theorem 4.10]{ShelahClassificationTheory}. For additional accounts, also see~\cite[\S4]{adlerNIP} or~\cite[2.3.4]{simonNIP}.

\medskip\noindent
In the presence of the Generalized Continuum Hypothesis (GCH), we have $\ded(\kappa) = 2^{\kappa}$ for all $\kappa$ and so we cannot get a converse to Proposition~\ref{stabfunNIP}. However, if we dare to reject CH, then we have~\cite[Corollary 4.3]{mitchell} at our disposal:
\begin{prop}
$\operatorname{Con}(\text{ZF}) \to \operatorname{Con}(\text{ZFC}, \;\;2^{\aleph_0} = \aleph_{\omega_1}, \;\;2^{\aleph_1}=\aleph_{\omega_1}^+,\text{ and}\;\;\ded(\aleph_1)<2^{\aleph_1}).$
\end{prop}

\noindent
Note that if we are in a model of ZFC where $2^{\aleph_0} = \aleph_{\omega_1}$, $2^{\aleph_1}=\aleph_{\omega_1}^+$ and $\ded(\aleph_1)<2^{\aleph_1}$ are true, then it follows that $\ded(\aleph_1)\leq\aleph_{\omega_1}$ and so
\[
\ded(\aleph_1)^{\aleph_0}\leq \aleph_{\omega_1}^{\aleph_0} = (2^{\aleph_0})^{\aleph_0} = 2^{\aleph_0\cdot\aleph_0} = 2^{\aleph_0} = \aleph_{\omega_1}<\aleph_{\omega_1}^{+} = 2^{\aleph_1}.
\]

\medskip\noindent
In other words:

\begin{cor}[Mitchell]
\label{keycon}
$\operatorname{Con}(\text{ZF})\to \operatorname{Con}(\text{ZFC and $\ded(\aleph_1)^{\aleph_0}<2^{\aleph_1}$})$.
\end{cor}

\medskip\noindent
By absoluteness of NIP, Proposition~\ref{stabfunNIP} and Corollary~\ref{keycon}, we get:

\begin{proof}[Proof of Theorem~\ref{TloghasNIP}]
Since $T_{\log}$ is countable in a recursive language with a recursively enumerable axiomatization, the statement ``$T_{\log}$ has NIP" is an arithmetic statement, i.e., via G\"{o}del numbering this statement is expressible by a sentence in Peano arithmetic. Any proof of such a sentence from ZFC $ + (\ded(\aleph_1)^{\aleph_0}<2^{\aleph_1})$ can be converted into a (possibly much longer) proof from ZFC. Now, suppose we are in a model of ZFC $ + (\ded(\aleph_1)^{\aleph_0}<2^{\aleph_1})$. Then in such a model it follows from Corollary~\ref{countingtypescor} that $g_{T_{\log}}(\aleph_1)\leq \ded(\aleph_1)^{\aleph_0}<2^{\aleph_1}$. Then by Proposition~\ref{stabfunNIP}, it follows that $T_{\log}$ has NIP in that particular model, i.e.
\[
\text{ZFC $ + (\ded(\aleph_1)^{\aleph_0}<2^{\aleph_1})\vdash \text{$T_{\log}$ has NIP}$}
\]
and thus 
\[
\text{ZFC $ \vdash\text{$T_{\log}$ has NIP}$,}
\] or in other words, $T_{\log}$ has NIP. It follows that $T_0$ also has NIP since every model of $T_0$ can be expanded into a model of $T_{\log}$.
\end{proof}

\section{Other Results}
\label{otherresults}

\subsection{The Steinitz Exchange Property} Given an arbitrary theory $T$, a parameter set $A$ and an element $a$ in $\M$, we say that \textbf{$a$ is algebraic over $A$} if $a$ belongs to a \emph{finite} $A$-definable subset of $\M$. Then we define the \textbf{algebraic closure of $A$ in $\M$} as the set
\[
\acl(A):= \{a\in\M: \text{$a$ is algebraic over $A$}\}.
\]

\begin{definition}
A theory $T$ is said to have the \textbf{Steinitz exchange property} if for all sets $A$ and all elements $a,b\in\M$, if $a\not\in\acl(A)$ and $b\not\in\acl(A)$, then
\[
a\in\acl(A\cup\{b\}) \Longleftrightarrow b\in\acl(A\cup\{a\}).
\]
\end{definition}

\noindent
If a theory $T$ has the Steinitz exchange property, then the algebraic closure operator $\acl$ will be a so-called \emph{pregeometry}. For more on the role of pregeometries in model theory, we refer the reader to~\cite[Chapter 8]{marker}. For our theory $T_{\log}$, the algebraic closure operator will \emph{not} be a pregeometry:

\begin{prop}
$T_{\log}$ does not have the Steinitz exchange property.
\end{prop}
\begin{proof}
Since $T_{\log}$ has a universal axiomatization and is model complete, we have that for all $A$, $\acl(A) = \langle A\rangle$. Let $\Gamma$ be a small model and construct an elementary extension $\Gamma_{\rho}$ of $\Gamma$ for some $\rho:2\to\sded(\Psi)$ inside $\M$. Let $(\beta_{k,0})$ and $(\beta_{k,1})$ be the two copies of $\Z$ which were added to $\Gamma$ in $\Gamma_{\rho}$. Let $a = \beta_{0,0}$ and $b = \beta_{0,0}+\beta_{0,1}$. By calculations done in Section~\ref{examples}, we have $\acl(\Gamma \cup\{b\}) = \Gamma\langle b\rangle = \Gamma_{\rho}$ whereas $\acl(\Gamma \cup\{a\}) = \Gamma\langle a\rangle = \Gamma_{\rho\upharpoonright 1}$.
\end{proof}

\subsection{Applications of Section~\ref{moreasymptoticintegration}} 
\emph{In this subsection we let $(\Gamma,\psi)$ be a divisible $H$-asymptotic couple with asymptotic integration, construed as an $L_{AC}$-structure in the obvious way.}
 We let $\chi$ denote the contraction map on $(\Gamma,\psi)$. The material in this subsection naturally would belong in Section~\ref{moreasymptoticintegration} as it applies in general to arbitrary $H$-asymptotic couples with asymptotic integration. However, we chose to relegate it to Section~\ref{otherresults} because it was not relevant for Section~\ref{sectionsimpleextensions} and because of its relevance in the next subsection.

\medskip\noindent
We begin with the following application of Proposition~\ref{spsigap}:

\begin{cor}
\label{sequalspsi}
For any $q\in\Q^{>}$, and for all $\alpha\in\Gamma$ such that $|\alpha|>(1+q)|s0|$, $s(\alpha) = \psi(\alpha)$.
\end{cor}

\noindent
Corollary~\ref{sequalspsi} and its proof below indicates the functions $s$ and $\psi$ agree sufficiently far away from the convex hull of $\{0\}\cup\{s0\}\cup\Psi^{>s0}$. At the moment this observation isn't very fruitful for models of $T_0$ since most of the action happens around this set anyway. However, for other asymptotic couples, such as the so-called \emph{closed asymptotic couples} of~\cite{closedasymptoticcouples}, this can be useful in further relating the roles of $s$ and $\psi$.

\medskip
\noindent
We begin first with a lemma which further clarifies the relationship between $s0$, $0$ and $\Psi$ in an H-asymptotic couple:
\begin{lemma}
$s0\neq 0$ and thus either $s0<0$ or $0<s0$. If $s0<0$, then $\Psi<(1-q)s0$, and if $0<s0$, then $\Psi<(1+q)s0$ for any $q\in\Q^{>}$.
\end{lemma}
\begin{proof}
Since $\int0\neq 0$, it follows that $0-\int 0\neq 0$ and thus $s0\neq 0$. If $s0<0$, then 
\[
(-qs0)' = -qs0+\psi(-qs0) = -s0+\psi(s0) = (1-q)s0\in(\Gamma^{>})'
\]
 and thus $\Psi<0$ by (AC3). If $0<s0$, then for $q\in\Q^{>}$, 
 \[
 (qs0)' = qs0+\psi(qs0) = (1+q)s0\in(\Gamma^{>})',
 \]
  and likewise $\Psi<(1+q)s0$.
\end{proof}

\begin{proof}[Proof of Corollary~\ref{sequalspsi}]
Suppose $q\in\Q^{>}$ and $\alpha\in\Gamma$ is such that $|\alpha|>(1+q)|s0|$. Let $(\Gamma^*,\psi^*)$ be an $H$-asymptotic couple with asymptotic integration that extends $(\Gamma,\psi)$ which contains an element $\gamma^*\in\Psi^*$ such that $\Psi<\gamma^*$. If $s0<0$, then $s0<\gamma^*<0$ and thus $|\alpha|\geq (1+q)|\gamma^*|$. Otherwise, if $s0>0$, then $s0<\gamma^*<(1+q')s0$ for every $q'\in\Q^{>}$ and thus $|\alpha|\geq (1+q)|\gamma^*|$ as well in this case. In both cases, $[\alpha-\gamma^*] = [\alpha]$ and thus $s(\alpha) = \psi^*(\alpha-\gamma^*) = \psi^*(\alpha) =  \psi(\alpha)$ by Proposition~\ref{spsigap}.
\end{proof}

\medskip\noindent
As another application of $s$-cuts, Definition-Lemma~\ref{Balphashift} below gives a method of producing a new $\psi$-map from an old $\psi$-map, while keeping the underlying ordered divisible abelian group and original contraction map the same. Recall that $\chi+\psi\circ\chi = \psi$ is the defining relation for the contraction map $\chi$ on $\Gamma^{<}$ in the asymptotic couple $(\Gamma,\psi)$.

\begin{definitionlemma}
\label{Balphashift}
Let $B\in\sded(\Psi)$ and $\epsilon\in\Gamma$ be such that $\psi(\epsilon)\in B$. Define the \textbf{$(B,\epsilon)$-shift} of $\psi$ to be the function $\widetilde{\psi}:\Gamma_{\infty}\to\Gamma_{\infty}$ such that
\[
\widetilde{\psi}(\alpha) = \begin{cases} \psi(\alpha) & \text{if $\psi(\alpha)<B$} \\ \psi(\alpha)+\epsilon & \text{if $\psi(\alpha)\in B$} \\ \infty & \text{if $\alpha = 0$.} \end{cases}
\]
Then $(\Gamma,\widetilde{\psi})$ is a divisible $H$-asymptotic couple with asymptotic integration such that $\chi+\widetilde{\psi}\circ\chi = \widetilde{\psi}$ on $\Gamma^{<}$.
\end{definitionlemma}
\begin{proof} We'll first show (HC). Suppose $0<\alpha<\beta$ and $\psi(\alpha)\in B$ and $\psi(\beta)<B$. Then by Corollary~\ref{successorid}, $\psi(\psi(\beta)-\psi(\alpha)) = s\psi(\beta)<B$ whereas $\psi(\epsilon)\in B$. By (HC) for $(\Gamma,\psi)$, it follows that $[\epsilon]<[\psi(\alpha)-\psi(\beta)]$ and thus $\psi(\alpha)-\psi(\beta)\geq -\epsilon$ since $\psi(\alpha)-\psi(\beta)>0$.
From this we get $\widetilde{\psi}(\alpha) = \psi(\alpha)+\epsilon\geq \psi(\beta) = \widetilde{\psi}(\beta)$. All other cases are trivial. 

(AC2) is clear. 

For (AC1), first suppose that $\alpha,\beta$ are such that $[\alpha]>[\beta]$. Then $\widetilde{\psi}(\alpha+\beta) = \widetilde{\psi}(\alpha) \geq \min(\widetilde{\psi}(\alpha),\widetilde{\psi}(\beta))$ by (HC) and (AC2). Otherwise, assume that $[\alpha]=[\beta]$ and $\psi(\alpha)=\psi(\beta)<B$ and $\psi(\alpha+\beta)\in B$. Then by a similar argument as for (HC) using $[\epsilon]<[\psi(\alpha+\beta)-\psi(\alpha)]$, we can show that $\widetilde{\psi}(\alpha+\beta) = \psi(\alpha+\beta)+\epsilon\geq \psi(\alpha) = \min(\widetilde{\psi}(\alpha),\widetilde{\psi}(\beta))$. All other cases are trivial.

Instead of verifying (AC3), by~\cite[Lemma 6.5.5]{mt} it is sufficient to show that the map $\gamma\mapsto \gamma+\widetilde{\psi}(\gamma):\Gamma^{>}\to\Gamma$ is strictly increasing. The main case to consider is $0<\alpha<\beta$ where $\psi(\alpha)\in B$ and $\psi(\beta)<B$. In this case, $[\beta]>[\alpha],[\epsilon]$ and so
\[
\psi(\alpha)< (\beta-\alpha-\epsilon)' = \beta-\alpha-\epsilon + \psi(\beta-\alpha-\epsilon) = \beta-\alpha-\epsilon+\psi(\beta)
\]
by (HC) and (AC3) for $(\Gamma,\psi)$. Rearranging terms gives us $\alpha+\psi(\alpha)+\epsilon<\beta+\psi(\beta)$, or rather $\alpha+\widetilde{\psi}(\alpha)<\beta+\widetilde{\psi}(\beta)$.

To show that $(\Gamma,\widetilde{\psi})$ has asymptotic integration, let $\widetilde{\Psi} := \widetilde{\psi}(\Gamma^{\neq})$. Suppose towards a contradiction that there is $\gamma\in\Gamma$ such that $\gamma = \sup\widetilde{\Psi}$. Since $\widetilde{\psi}(B)$ is cofinal in $\widetilde{\Psi}$, we have that $\gamma = \sup\widetilde{\psi}(B) = \sup \psi(B)+\epsilon = \sup \Psi+\epsilon$. Thus $\gamma-\epsilon = \sup \Psi$, a contradiction because $(\Gamma,\psi)$ has asymptotic integration.

For the claim about the contraction mapping, note that for all $\alpha\in\Gamma$, $\psi(\chi(\alpha)) = s\psi(\alpha)$. Thus $\psi(\alpha)<B$ iff $\psi(\chi(\alpha))<B$.
\end{proof}

\noindent
As a special case of Definition-Lemma~\ref{Balphashift}, we note that the $(\Psi,\epsilon)$-shift of $\psi$ is just a shift $(\Gamma,\psi+\epsilon)$ in the sense of~\cite[Pg. 978, Lemma(2)]{differentialvaluationII}. See also~\cite[\S6.5]{mt}.

\medskip\noindent
In general, if $(\Gamma,\widetilde{\psi})$ is a $(B,\epsilon)$-shift of $(\Gamma,\psi)$, then we do not expect these asymptotic couples, as $L_{AC}$-structures, to be elementarily equivalent. Indeed, if $(\Gamma,\psi)\models T_0$, then the $(\Psi,-s0)$-shift $(\Gamma,\widetilde{\psi})$ will not be a model of $T_0$ because $\min\widetilde{\Psi} = 0$ in that case. However, we do have the following:

\begin{prop}
Suppose $(\Gamma,\psi)\models T_0$ and $B\in\sded(\Psi)$ is such that $B\neq\Psi$ and $\epsilon\in\Gamma$ is such that $\psi(\epsilon)\in B$. Then the $(B,\epsilon)$-shift $(\Gamma,\widetilde{\psi})$ is also a model of $T_0$.
\end{prop}
\begin{proof}
$(\Gamma,\widetilde{\psi})$ is a divisible $H$-asymptotic couple with asymptotic integration such that $\chi+\widetilde{\psi}\circ\chi = \widetilde{\psi}$. Let $\widetilde{s}$ be the successor function of $(\Gamma,\widetilde{\psi})$. It is clear that $\widetilde{\Psi}$ is a successor set with least element $s0 = \widetilde{s}>0$, since the order types of $\Psi$ and $\widetilde{\Psi}$ are the same and these $\Psi$-sets have at least the first copy of $\N$ in common.

\begin{claim}
Suppose $\alpha$ is such that $\psi(\alpha)\in B$. Then $\widetilde{s}(\widetilde{\psi}(\alpha)) = s\psi(\alpha)+\epsilon$.
\end{claim}
\begin{proof}[Proof of Claim]
By the relation $s\psi = \psi\chi$, which holds in every $H$-asymptotic couple with asymptotic integration, and the fact that $\widetilde{\chi} = \chi$, we have
\[
\widetilde{s}(\widetilde{\psi}(\alpha)) = \widetilde{\psi}(\widetilde{\chi}(\alpha)) = \widetilde{\psi}(\chi(\alpha)) = \psi(\chi(\alpha))+\epsilon = s(\psi(\alpha))+\epsilon.\qedhere
\]
\end{proof}
By the claim it follows that each $\alpha\in\widetilde{\Psi}$ has immediate successor $\widetilde{s}(\alpha)$ and that $\gamma\mapsto \widetilde{s}\gamma:\widetilde{\Psi}\to\widetilde{\Psi}^{>s0}$ is a bijection.
\end{proof}

\subsection{Relation to Precontraction Groups}

In this subsection we will make a remark about the relationship between our asymptotic couples and the precontraction groups of Kuhlmann. Precontraction groups arise as the value groups of certain ordered exponential fields, and in this way they are similar in spirit to asymptotic couples which arise as the value groups of certain valued differential fields. We refer the interested reader to~\cite{kuhlmann1,kuhlmann2} for a treatment of the model theory of precontraction groups and to~\cite{SKuhlmann} for their connection to ordered exponential fields. For our purposes, it suffices to recall the definition:

\begin{definition}
A \textbf{precontraction group} is a pair $(\Gamma,\chi)$ where $\Gamma$ is an ordered abelian group and $\chi:\Gamma\to\Gamma$ satisfies for all $\alpha,\beta\in\Gamma$:
\begin{enumerate}
\item $\chi(\alpha) = 0 \Longleftrightarrow \alpha = 0$;
\item $\alpha\leq\beta\Longrightarrow \chi(\alpha)\leq\chi(\beta)$;
\item $\chi(-\alpha) = -\chi(\alpha)$;
\item $[\alpha] = [\beta]$ and $\sign(\alpha) = \sign(\beta) \Longrightarrow \chi(\alpha) = \chi(\beta)$.
\end{enumerate}
If in addition, for all $\alpha\in\Gamma^{\neq}$:
\begin{enumerate}
  \setcounter{enumi}{4}
\item $|\alpha|>|\chi(\alpha)|$
\end{enumerate}
then $(\Gamma,\chi)$ is said to be a \textbf{centripetal} precontraction group. Finally, we say that a precontraction group $(\Gamma,\chi)$ is \textbf{divisible} if the underlying ordered abelian group $\Gamma$ is divisible.

\medskip\noindent
We let $L_{PG} = \{0,+,-,<,\chi\}$ denote the natural first-order language of precontraction groups and construe all precontraction groups $(\Gamma,\chi)$ as $L_{PG}$-structures in the obvious way.
\end{definition}

\medskip\noindent
If $(\Gamma,\psi)$ is a divisible $H$-asymptotic couple with asymptotic integration, then we may associate to $(\Gamma,\psi)$ a divisible centripetal precontraction group $(\Gamma,\chi_{PG})$ by defining for all $\alpha\in\Gamma$,
\[
\chi_{PG}(\alpha) = \begin{cases}
\chi(\alpha) & \text{if $\alpha<0$,} \\
0 & \text{if $\alpha = 0$,} \\
-\chi(-\alpha) & \text{if $\alpha>0$,}
\end{cases}
\]
where $\chi = \int\psi:\Gamma^{<}\to\Gamma^{<}$ is the contraction map of $(\Gamma,\psi)$ as defined in Definition~\ref{functionsdefs}. Thus every divisible $H$-asymptotic couple with asymptotic integration yields a divisible centripetal precontraction group as a reduct. Conversely, it is worth considering whether this process is reversible, i.e., given a divisible centripetal precontraction group $(\Gamma,\chi_{PG})$, can one define a $\psi$-map on $\Gamma$ in the $L_{PG}$-structure $(\Gamma,\chi_{PG})$ such that $(\Gamma,\psi)$ is a divisible $H$-asymptotic couple with asymptotic integration and such that the contraction map of $(\Gamma,\psi)$ is $\chi_{PG}|\Gamma^{<}$. It turns out this is impossible for models of $T_0$:

\begin{prop}
\label{precontractionprop}
In no precontraction group $(\Gamma,\chi)$ can one define, even allowing parameters, a function $\psi:\Gamma^{\neq}\to\Gamma$ such that $(\Gamma,\psi)$ is a model of $T_0$ and $\chi+\psi\circ\chi=\psi$ on $\Gamma^{<}$.
\end{prop}
\begin{proof}
Suppose $(\Gamma,\psi)\models T_0$ is such that we can define $\psi$ in $(\Gamma,\chi)$. We may assume that $(\Gamma,\psi)$ is $\aleph_0$-saturated. Take $B\in\sded(\Psi)$ large enough so that it is to the right of the $\Psi$-set of the definable closure of all the finitely-many parameters needed from $\Gamma$ to define $\psi$ in $(\Gamma,\chi)$. Consider any $(B,\epsilon)$-shift $\widetilde{\psi}$ of $\psi$ such that $\psi(\epsilon)\in B$. Then $(\Gamma,\psi)\equiv (\Gamma,\widetilde{\psi})$ and $(\Gamma,\chi) = (\Gamma,\widetilde{\chi})$. By completeness of $T_0$, the same formula that defines $\psi$ in $(\Gamma,\chi)$ must define $\widetilde{\psi}$ in $(\Gamma,\widetilde{\chi})$ and so $\psi = \widetilde{\psi}$, a contradiction.
\end{proof}

\medskip\noindent
Our method of proof for Proposition~\ref{precontractionprop} mirrors the proof given in~\cite[Prop 5.1]{someremarks} for the corresponding result about closed asymptotic couples. A \textbf{closed asymptotic couple} is a divisible $H$-asymptotic couple with asymptotic integration such that $(\Gamma^{<})' = \Psi$ (see~\cite{closedasymptoticcouples}). There they use essentially the same trick with $(B,\epsilon)$-shifts, except they consider iterates of $\psi$ instead of iterates of $s$. However, by Corollary~\ref{sequalspsi}, one can see that this is essentially the same notion for elements $\alpha\ll 0$.

\medskip\noindent
Furthermore, it seems likely that this trick can be used for any theory $\Th(\Gamma,\psi)$ of interest, where $(\Gamma,\psi)$ is a divisible $H$-asymptotic couple with asymptotic integration. Provided that the first order theory of $(\Gamma,\psi)$ is preserved under sufficiently subtle $(B,\epsilon)$-shifts, the same proof can be used. This leads us to the following:

\begin{conjecture}
\label{precontractiongroupconjecture}
In no nontrivial precontraction group $(\Gamma,\chi)$ can one define, even allowing parameters, a function $\psi:\Gamma^{\neq}\to\Gamma$ such that $(\Gamma,\psi)$ is an $H$-asymptotic couple and $\chi+\psi\circ\chi = \psi$ on $\Gamma^{<}$.  
\end{conjecture}

%
%

\section{Conclusion}
\label{conclusion}
\noindent
We conclude with a list of unresolved issues and things left to do:

\begin{enumerate}
\item Settle Conjecture~\ref{precontractiongroupconjecture}.
\item Describe all definable functions $\Gamma\to\Gamma_{\infty}$, where $(\Gamma,\psi)$ is a model of $T_{\log}$.
\item Give a more concrete proof of NIP for $T_{\log}$ which avoids an absoluteness argument.
\item Is $T_{\log}$ \emph{distal}? Distal theories form a subclass of NIP theories which in some sense are purely unstable. See~\cite{distal} for a definition of distality.
\item Is $(\Gamma,\psi)$ \emph{quasi-weakly-o-minimal}, i.e., any definable subset is a finite boolean combination of convex sets and $0$-definable sets? For more information on this property in the o-minimal setting, see~\cite{quasiominimal}.
\item Is $(\Gamma,\psi)$ \emph{d-minimal}, i.e., any definable subset of $\Gamma$ is a union of an open set and finitely many discrete sets? See~\cite[\S3.4]{dminimal} for a discussion of d-minimality in the context of expansions of the real field.
\end{enumerate}

\section*{Acknowledgements}
\noindent
The author would like to thank Franz-Viktor Kuhlmann for his hospitality during a visit to University of Saskatchewan in the late Fall of 2014 and for suggesting the author look into the Steinitz exchange property. The investigation into the relationship with precontraction groups also arose from discussions with Kuhlmann and Koushik Pal during that visit.
The author would also like to thank Justin Moore and Chris Laskowski for their assistance in navigating some of the consistency issues arising in Section~\ref{NIPsettheory}, and additionally thanks Laskowski for the invitation to visit University of Maryland, College Park over Thanksgiving 2014.
Above all, the author would like to thank Lou van den Dries for his guidance, encouragement, and numerous discussions around the topics of this paper and to Qingci An, Jacob Harris and Konrad Wrobel for the illustrations. Finally, the author would like to thank the referee for the many helpful comments and suggestions.

\bibliographystyle{amsalpha}	
\bibliography{refs}

\providecommand{\bysame}{\leavevmode\hbox to3em{\hrulefill}\thinspace}
\providecommand{\MR}{\relax\ifhmode\unskip\space\fi MR }
\providecommand{\MRhref}[2]{%
  \href{http://www.ams.org/mathscinet-getitem?mr=#1}{#2}
}
\providecommand{\href}[2]{#2}
\begin{thebibliography}{AvdDvdH15}

\bibitem[Adl08]{adlerNIP}
Hans Adler, \emph{Introduction to theories without the independence property},
  \url{http://www.logic.univie.ac.at/~adler/docs/nip.pdf}, June 2008.

\bibitem[Asc03]{someremarks}
Matthias Aschenbrenner, \emph{Some remarks about asymptotic couples}, Valuation
  theory and its applications, {V}ol. {II} ({S}askatoon, {SK}, 1999), Fields
  Inst. Commun., vol.~33, Amer. Math. Soc., Providence, RI, 2003, pp.~7--18.
  \MR{2018547 (2004j:03043)}

\bibitem[AvdD00]{closedasymptoticcouples}
Matthias Aschenbrenner and Lou van~den Dries, \emph{Closed asymptotic couples},
  J. Algebra \textbf{225} (2000), no.~1, 309--358. \MR{1743664 (2001g:03065)}

\bibitem[AvdDvdH15]{mt}
Matthias Aschenbrenner, Lou van~den Dries, and Joris van~der Hoeven,
  \emph{Asymptotic differential algebra and model theory of transseries}, arXiv
  preprint arXiv:1509.02588 (2015), 703 pp.

\bibitem[BPW00]{quasiominimal}
Oleg Belegradek, Ya'acov Peterzil, and Frank Wagner, \emph{Quasi-o-minimal
  structures}, J. Symbolic Logic \textbf{65} (2000), no.~3, 1115--1132.
  \MR{1791366 (2001k:03079)}

\bibitem[CG96]{conway}
John~H. Conway and Richard~K. Guy, \emph{The book of numbers}, Copernicus, New
  York, 1996. \MR{1411676 (98g:00004)}

\bibitem[Geh14]{gehret}
Allen Gehret, \emph{The asymptotic couple of the field of logarithmic
  transseries}, arXiv preprint arXiv:1405.1012 (2014), (submitted).

\bibitem[Jec03]{jech}
Thomas Jech, \emph{Set theory}, Springer Monographs in Mathematics,
  Springer-Verlag, Berlin, 2003, The third millennium edition, revised and
  expanded. \MR{1940513 (2004g:03071)}

\bibitem[Kuh94]{kuhlmann1}
Franz-Viktor Kuhlmann, \emph{Abelian groups with contractions. {I}}, Abelian
  group theory and related topics ({O}berwolfach, 1993), Contemp. Math., vol.
  171, Amer. Math. Soc., Providence, RI, 1994, pp.~217--241. \MR{1293144
  (95i:03079)}

\bibitem[Kuh95]{kuhlmann2}
\bysame, \emph{Abelian groups with contractions. {II}. {W}eak {${\rm
  o}$}-minimality}, Abelian groups and modules ({P}adova, 1994), Math. Appl.,
  vol. 343, Kluwer Acad. Publ., Dordrecht, 1995, pp.~323--342. \MR{1378210
  (97g:03044)}

\bibitem[Kuh00]{SKuhlmann}
Salma Kuhlmann, \emph{Ordered exponential fields}, Fields Institute Monographs,
  vol.~12, American Mathematical Society, Providence, RI, 2000. \MR{1760173
  (2002m:12004)}

\bibitem[Kun80]{kunen}
Kenneth Kunen, \emph{Set theory}, Studies in Logic and the Foundations of
  Mathematics, vol. 102, North-Holland Publishing Co., Amsterdam-New York,
  1980, An introduction to independence proofs. \MR{597342 (82f:03001)}

\bibitem[Mar02]{marker}
David Marker, \emph{Model theory}, Graduate Texts in Mathematics, vol. 217,
  Springer-Verlag, New York, 2002, An introduction. \MR{1924282 (2003e:03060)}

\bibitem[Mil05]{dminimal}
Chris Miller, \emph{Tameness in expansions of the real field}, Logic
  {C}olloquium '01, Lect. Notes Log., vol.~20, Assoc. Symbol. Logic, Urbana,
  IL, 2005, pp.~281--316. \MR{2143901 (2006j:03049)}

\bibitem[Mit73]{mitchell}
William Mitchell, \emph{Aronszajn trees and the independence of the transfer
  property}, Ann. Math. Logic \textbf{5} (1972/73), 21--46. \MR{0313057 (47
  \#1612)}

\bibitem[Ros79]{differentialvaluation1}
Maxwell Rosenlicht, \emph{On the value group of a differential valuation},
  Amer. J. Math. \textbf{101} (1979), no.~1, 258--266. \MR{527836}

\bibitem[Ros80]{differentialvaluations}
\bysame, \emph{Differential valuations}, Pacific J. Math. \textbf{86} (1980),
  no.~1, 301--319. \MR{586879}

\bibitem[Ros81]{differentialvaluationII}
\bysame, \emph{On the value group of a differential valuation. {II}}, Amer. J.
  Math. \textbf{103} (1981), no.~5, 977--996. \MR{630775 (83d:12013)}

\bibitem[She90]{ShelahClassificationTheory}
S.~Shelah, \emph{Classification theory and the number of nonisomorphic models},
  second ed., Studies in Logic and the Foundations of Mathematics, vol.~92,
  North-Holland Publishing Co., Amsterdam, 1990. \MR{1083551}

\bibitem[{Sim}12]{simonNIP}
P.~{Simon}, \emph{{A Guide to NIP theories}}, ArXiv e-prints (2012).

\bibitem[Sim13]{distal}
Pierre Simon, \emph{Distal and non-distal {NIP} theories}, Ann. Pure Appl.
  Logic \textbf{164} (2013), no.~3, 294--318. \MR{3001548}

\end{thebibliography}

\end{document}